\documentclass[a4paper,11pt]{amsart}
\usepackage{amsmath,inputenc,euscript,amssymb,geometry}
\usepackage{graphicx}
\usepackage{advdate}
\usepackage[normalem]{ ulem }
\usepackage{soul}
\usepackage{enumerate}
\usepackage{enumitem}
\usepackage{amssymb}
\usepackage{latexsym}
\usepackage{amssymb,amsbsy,amsmath,amsfonts,amssymb,amscd}
\usepackage{hyperref}
\usepackage{xcolor}

\newtheorem{lemma}{Lemma}[section]
\newtheorem{theorem}[lemma]{Theorem}

\begin{document}
\title[Turbulent solutions of the binormal flow and the 1D cubic NLS]{Turbulent solutions of the binormal flow and\\ the 1D cubic Schr\"odinger equation}
\author[V. Banica]{Valeria Banica}
    \address[V. Banica]{Sorbonne Universit\'e, Universit\'e Paris Cit\'e, CNRS, INRIA, Laboratoire Jacques-Louis Lions, LJLL, F-75005 Paris, France, Valeria.Banica@sorbonne-universite.fr} 
\author[L. Vega]{Luis Vega}
\address[L. Vega]{BCAM-UPV/EHU Department of Mathematics, University of Basque Country, Apdo 644, 48080 Bilbao, Spain, luis.vega@ehu.es} 
\date{\today}

\maketitle

\begin{abstract}
In the last three decades there has been an intense activity on the exploration of turbulent phenomena of dispersive equations, as for instance the growth of Sobolev norms since the work of Bourgain in the 90s. In general the 1D cubic Schr\"odinger equation has been left aside because of its complete integrability. In a series of papers of the last six years  that we survey here for the special issue of the ICMP 2024 (\cite{BanicaVega2020},\cite{BanicaVega2020bis},\cite{BanicaVega2022},\cite{BanicaVega2022bis},\cite{BanicaVega2024},\cite{BanicaEceizabarrenaNahmodVega2024},\cite{BanicaLucaTzvetkovVega2024}), we considered, together with the 1D cubic Schr\"odinger equation, the binormal flow, which is a geometric flow explicitly related to it. We displayed rigorously a large range of complex behavior as creation of singularities and unique continuation, Fourier growth, Talbot effects, intermittency and multifractality, justifying in particular some previous numerical observations. To do so we constructed a class of well-posedness for the 1D cubic Schr\"odinger equation included in the critical Fourier-Lebesgue space $\mathcal FL^\infty$ and in supercritical Sobolev spaces with respect to scaling. Last but not least we recall that the binormal flow is a classical model for the dynamics of a vortex filament in a 3D fluid or superfluid, and that vortex motions are a key element of turbulence. 
\end{abstract}

\date\today
\maketitle

\tableofcontents

\section{Introduction}

\subsection{Vortex filaments dynamics and the binormal flow}
The evolution of vortex filaments is a key element in fluid and superfluid turbulence.  
We consider the classical model of the binormal flow (BF, or LIA from ``local induction approximation" or VFE from ``vortex filament equation"), which is the formally derived model for one vortex filament dynamics in a 3D fluid governed by Euler equations, also used in superfluids. If the vorticity at time $t$ is concentrated along an arclength-parametrized curve $\chi(t)$ in $\mathbb R^3$, its evolution in time is expected to evolve according to the binormal flow:
\begin{equation}\label{VFE}\chi_t=\chi_x\wedge\chi_{xx}.
\end{equation}
By using the Frenet system (see Appendix \ref{section-Hasimoto}), that characterizes, in terms of the curvature and torsion, the tangent and normal vectors, together with their vectorial product which is called the binormal vector, the equation writes as:
$$\chi_t=c\,b,$$
where $c$ stands for the curvature and $b$ for the binormal vector. This explains the binormal flow name.  The model was derived formally by using Biot-Savart integral by Da Rios in 1906 following the works of Levi-Civita (\cite{DaRios1906}). This was justified rigorously by the ``if" theorem of Jerrard and Seis in 2017 (\cite{JerrardSeis2017}). More precisely, under the assumption that the vorticity is concentrated
along a smooth curve $\tilde\chi(t)$ in $\mathbb R^3$ for all times $t\in[0,T]$ for some $T$, and for initial velocity field having controlled excess of the kinetic energy relative to the initial curve, they prove that $\tilde\chi(t)$ evolves to leading order
by binormal curvature flow. The proof relies on estimates on the
distance of the Hamiltonian-Poisson structures of Euler equation and of the binormal flow, together with stability
estimates obtained previously by Jerrard and Smets (\cite{JerrardSmets2015}).
Understanding when the vorticity propagates its initial structure of being concentrated along a curve is still a very difficult open problem. 

We shall recall here only a representative few advances in this direction in the last five years. Concerning the Navier-Stokes equation the Cauchy problem was proved by Bedrossian, Germain and Harrop-Griffiths (\cite{BedrossianGermainHarropGriffiths2023}) to be locally well-posed for 
an initial filament data  with no symmetry assumptions, but for small times not allowing to observe the binormal flow and to pass to the vanishing viscosity limit.  Vanishing viscosity limit was proved by Gallay and Sverak (\cite{GallaySverak2024}) for the particular case of axisymmetric vortex rings. The binormal flow dynamic was recovered by Fontelos and Vega (\cite{FontelosVega2023}) for Giga-Miyakawa solutions with initial filament data with no symmetry assumptions, but this regime does not allow for passing to the vanishing viscosity limit for positive times. 
In Euler equations Donati, Lacave and Miot (\cite{DonatiLacaveMiot2024}), and previously D\'avila, Del Pino, Musso and Wei (\cite{DavilaDelPinoMussoWei2022}) and Cao and Wan (\cite{CaoWan2024}) by other methods, constructed solutions with vorticity concentrated on helices, which are particular solutions with helical symmetry of the binormal flow. Thus we are still facing a gap towards the case of no-symmetry vortex filaments in Euler equations.

\subsection{Links between the binormal flow and Schr\"odinger equations} We shall present now the link between the binormal flow and Schr\"odinger equations that is at the basis of our results. A detailed presentation is done in \S \ref{section-Hasimoto}. 
It is easy to see that if $\chi$ is a binormal flow solution then its tangent vector $T = \chi_x$ satisfies
\begin{equation}\label{Smap}T_t=T_x\wedge T_x+T\wedge T_{xx}=T\wedge T_{xx}.\end{equation}
This is the 1D Schr\"odinger map equation with values in the sphere $\mathbb S^2$, that coincides with the 1D Heisenberg continuous model derived in ferro-magnetic theory by Landau and Lifshitz in 1935 (\cite{LandauLifshitz1935}). 
Moreover, solutions of the Schr\"odinger map, and thus corresponding solutions of the binormal flow equation, are related, via the Hasimoto transformation, to the 1D 
focusing cubic Schr\"odinger equation (NLS) on the line:
\begin{equation}\label{CubicNLS}
iu_t+u_{xx}+ |u|^2u=0. 
\end{equation} 
Indeed Hasimoto discovered in 1972 (\cite{Hasimoto1972}) the following transform based on the Frenet system of curves given by tangent, normal and binormal vector. If $\chi$ is a binormal flow solution with non-vanishing curvature, it is easy to show\footnote{This is due to the fact that if $\chi(t,x)$ is a binormal flow solution then, by computing and identifying the crossed second derivatives of the tangent and the normal vector $T_{tx}=T_{xt},n_{tx}=n_{xt}$, it follows that its curvature and torsion $(c,\tau)$ satisfy the system (called intrinsic equations):
$$\left\{\begin{array}{ll}c_t=-2c_x\,\tau-c\,\tau_x,\\ 
\tau_t=\left(\frac{c_{xx}-c\,\tau^2}{c}+\frac {c^2}2\right)_x. \end{array}\right.$$
} that the following function, called filament function:
$$u(t,x):=c(t,x)e^{i\int_0^x\tau(t,s)ds},
$$
where $c,\tau$ are the curvature and torsion of the curve, satisfies the 1D Schr\"odinger equation
 \begin{equation}\label{Hasimoto}
iu_t+u_{xx}+\left(|u|^2-f\right)u=0,
\end{equation}
where $f$ is a space independent function determined by $(c,\tau)(t,0)$. The Hasimoto transform is thus assigning a solution of 1D Schr\"odinger equation to a solution of the binormal flow. 
In particular Hasimoto's transform can be seen as an inverse Madelung transform\footnote{From the intrinsic equations it follows that if $\chi$ solves the binormal flow then $c^2$ and $2\tau$ satisfy a Euler-Korteweg type equation. Thus Hasimoto's transform assigns to such Euler-Korteweg solutions a solution of 1D cubic Schr\"odinger equation. In particular, by a change of phase, Hasimoto's transform assigns a solution of \eqref{Hasimoto} with potential $f=1$, that is Gross-Pitaevskii equation.
Conversely, Madelung transform ensures that if $u$ is a solution of Gross-Pitaesvskii equation, then introducing $\rho,v$ such that $u(t,x)=\sqrt{\rho(t,x)}e^{i\theta(t)}e^{i\int_0^x v(t,s)ds}$, i.e. $\rho:=|u|^2$ and $v:=2\nabla_x Arg\, u$, then $(\rho,v)$ solve an Euler type equation with the extra quantum pressure, i.e. a Euler-Korteweg type equation.}. Conversely, for any function $f$ depending only on the time variable, for instance $f\equiv 0$, from a smooth solution $u$ of \eqref{Hasimoto}, Hasimoto gave a method to construct frames whose first vector is a solution of \eqref{Smap} and thus a solution of the binormal flow \eqref{VFE}. 

The non-vanishing curvature condition was removed by Koiso in 1997 (\cite{Koiso1997}) by considering instead of the filament function the complexified normal developement of the curve given by the coefficients appearing in the derivatives of parallel transport frames. 
For these frames the other vectors than the tangent are relatively parallel in the sense that their variation is in the direction of the tangent vector, see Bishop's article "There are more than one way to frame a curve" (\cite{Bishop1975}). In \S \ref{section-Hasimoto} we shall present in a detailed way the Hasimoto construction in this parallel transport frames framework. 

We note that using Hasimoto's approach one may try to find binormal flow solutions that generate singularities in finite time by considering smooth solutions of the 1D cubic Schr\"odinger equation that generate singularities in finite time. However, half of the job is to find a 1D cubic NLS solution with some precise description, smooth for example on $t>0$ and generating a singularity at $t=0$, and half of the job is to describe geometrically the associated binormal flow solution on $t>0$ in order to understand its behavior at $t=0$.

\subsection{Self-similar type solutions generating one singularity}\label{section-ssbf}
An important class of solutions of the binormal flow are the self-similar solutions, that are smooth curves which develop in finite time a singularity in the shape of a corner. More precisely, arclength parametrized curves solutions of binormal flow are preserved by the rescaling $\lambda^{-1}\chi(\lambda^{2} t,\lambda x)$. Therefore self-similar solutions of binormal flow are searched as
$\chi(t,x)=\sqrt{t}G(\frac x{\sqrt{t}})$. The profile curve $G(s)=\chi(1,s)$ and its Frenet frame satisfy:
$$\frac{G(s)}2-\frac{sG_s(s)}{2}=G_s\wedge G_{ss}\Rightarrow -\frac{sT_s(s)}{2}=T\wedge T_{ss}\Rightarrow -\frac{scn}2=T\wedge(c_sn-c^2T+c\tau b),$$
so this determines the curvature and torsion of the profile, $\tau(s)=\frac s2,c_s=0$. Going back to the self-similar variables, the self-similar solutions form a 1-parameter family $\{\chi_a\}_{a\in\mathbb R^{*+}}$ with curvature $c_a(t,x)=\frac a{\sqrt{t}}$ and torsion $\tau_a(t,x)=\frac x{2t}$. Note that the filament function is $a\frac{e^{i\frac{x^2}{4t}}}{\sqrt{t}}$ with initial value  $a\delta_0$, and it satisfies  \eqref{Hasimoto} with $f(t)=\frac{a^2}{t}$.
These solutions were known and used by physicists from the 80s in the framework of reconnection of vortex filaments in ferromagnetics and superfluids (Schwarz \cite{Schwarz1985}, Lakshmanan and Daniel \cite{LakshmananDaniel1981}, Buttke \cite{Buttke1988}). They were rigorously studied by Guti\'errez, Rivas and Vega in 2003 (\cite{GutierrezRivasVega2003}), who proved that a corner is generated at $t=0$ and the tangent vector of the profile curve $G_a$ has a limit at infinity:
$$\exists A^\pm_{a}\in\mathbb S^2,\quad T_a(1,x)\overset{x\rightarrow\pm\infty}{\longrightarrow} A^\pm_a.$$ 
Moreover, they 
proved that the value $\theta_a$ of the angle of the corner is related to the parameter $a$ by the nonlinear formula:
\begin{equation}\label{angle}
\sin\frac {\theta_a} 2=e^{-\frac{a^2}{2}}.
\end{equation}
This type of dynamics can be observed in fluids passing over a triangular obstacle and in fluid and superfluid vortex reconnection, see Figure 1. We also recall that the Constantin-Fefferman-Majda blow-up criterium (\cite{ConstantinFeffermanMajda1996}) on the variation of the direction of vorticity writes  $\int_0^t\|\nabla(\frac{\omega}{|\omega|})(\tau)\|_{L^\infty_x}^2d\tau=\infty$. Here, having in mind that the tangent vector models the direction of vorticity, see also evidence in this sense in Theorem 1 in \cite{FontelosVega2023}, we have 
$$\int_0^t\|\partial_x T_a(\tau)\|_{L^\infty}^2d\tau=\int_0^t\|c_a(\tau)\|_{L^\infty}^2d\tau=\int_0^t\frac {a^2}{\tau}d\tau=\infty.$$

\begin{center}
 \includegraphics[height=1.1in]{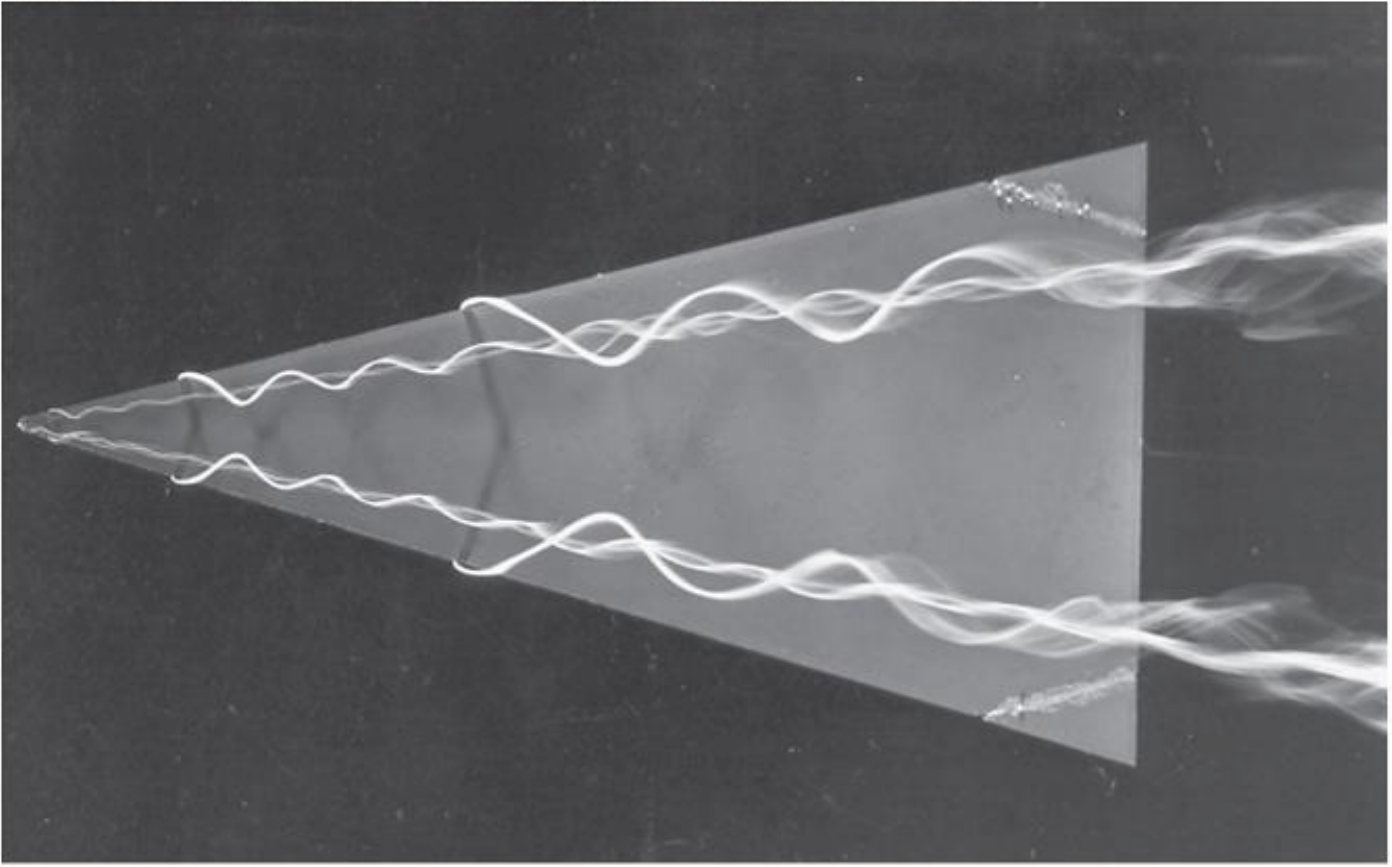} \quad\quad  \includegraphics[height=1.1in]{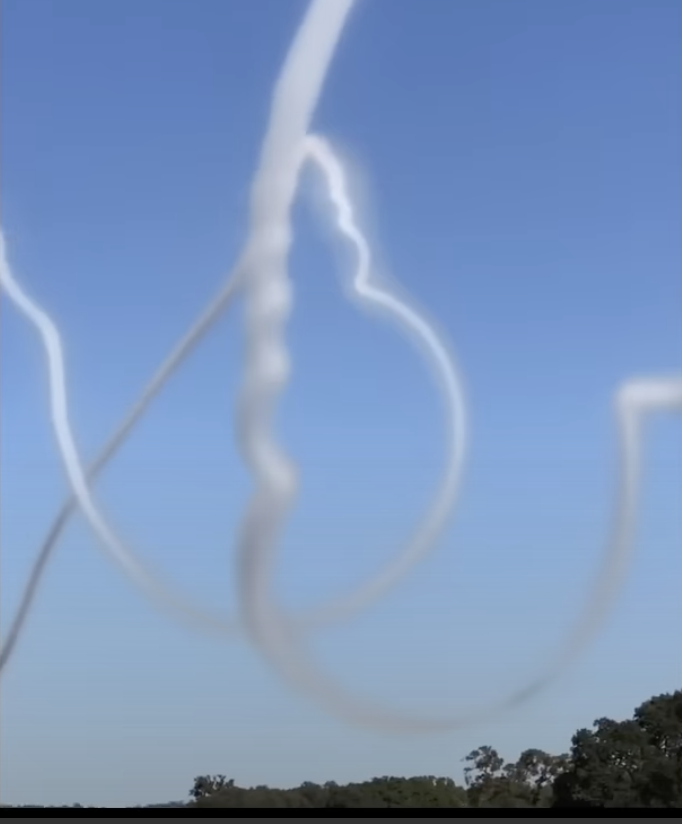}\quad\quad  \includegraphics[height=1.1in]{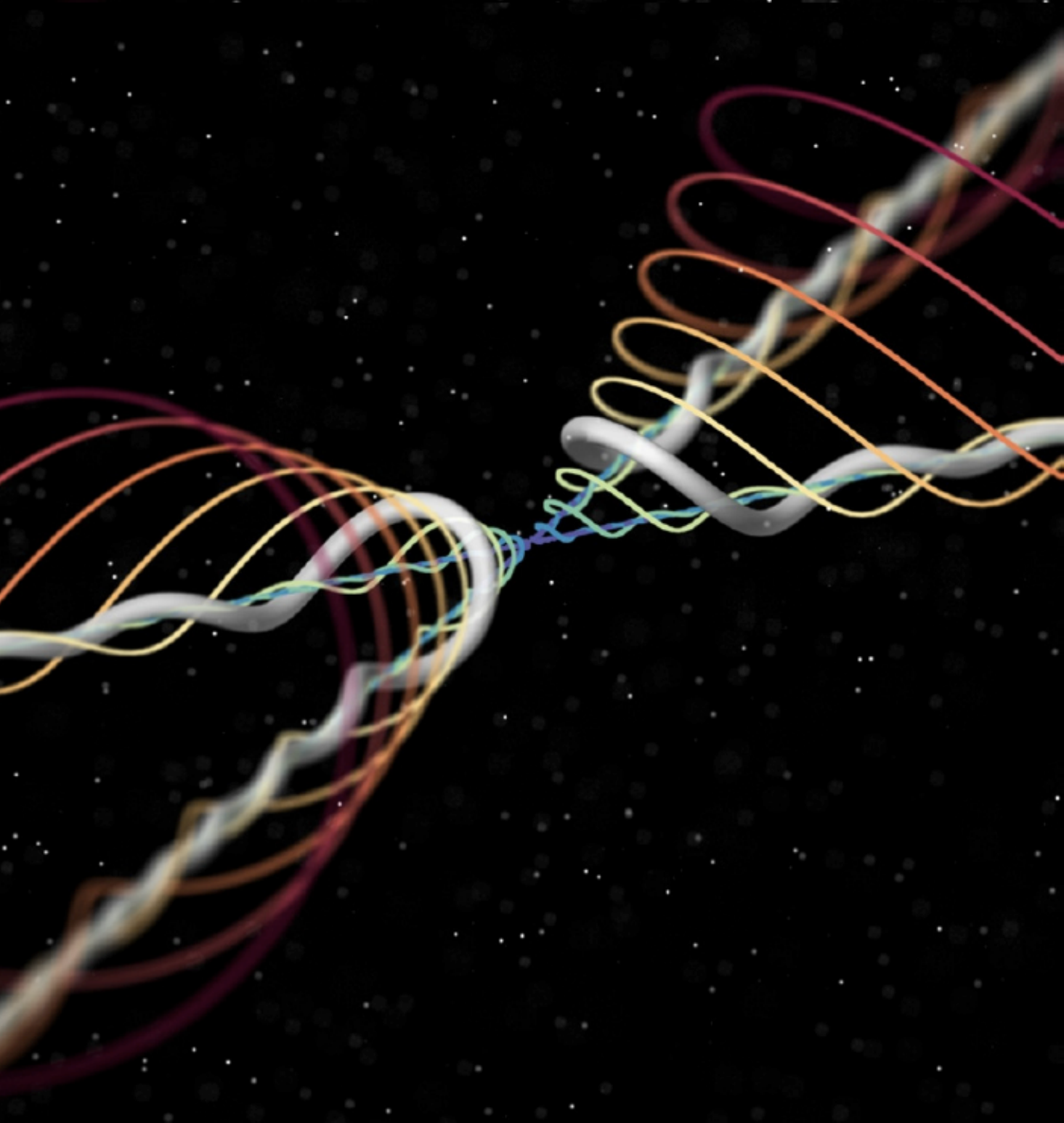}  \quad\quad  \includegraphics[height=1in]{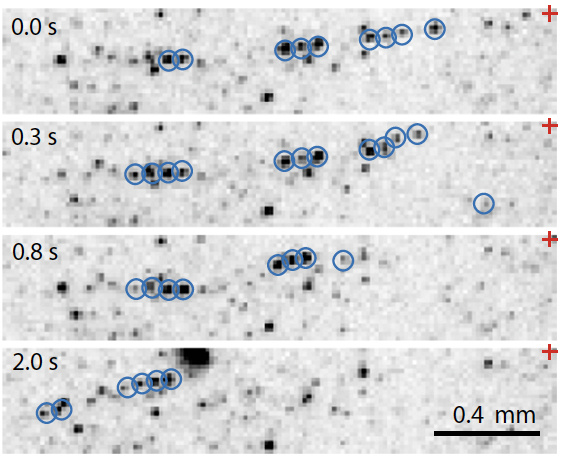}\\
{\small{Figure 1. Left: vortices in a fluid flowing over a triangular obstacle, Werl\'e, 1963.
  Middle left: picture of trail vortices after reconnection 2024. 
  Middle right: Numerical simulation of binormal flow selfsimilar solution (courtesy of Enrico Fonda 2014).
  Right: Direct observation of Kelvin waves excited by quantized vortex reconnection,
 Fonda et al. 2014 (\cite{FondaMeichleOuelletteLathrop2014}, see the two experiment movies on \href{https://www.pnas.org/doi/suppl/10.1073/pnas.1312536110}{https://www.pnas.org/doi/suppl/10.1073/pnas.1312536110}).}}
 \end{center}

 In a series of previous articles from 2009 to 2015 (\cite{BanicaVega2009},\cite{BanicaVega2012},\cite{BanicaVega2013},\cite{BanicaVega2015}, see also \cite{Guerin2023}) we understood more generally the behaviour of curves generating one corner in finite time and smoothing immediately after. As a first step we constructed solutions for \eqref{Hasimoto} with $f(t)=\frac{a^2}{t}$ that are small perturbations of $a\frac{e^{i\frac{x^2}{4t}}}{\sqrt{t}}$. Then these perturbations, that after pseudo-conformal transformation are long range scattering NLS solution, allowed us to construct via Hasimoto's method binormal flow solutions for $t> 0$, for which we determinate  the behavior at time $t=0$, and that we can continue for $t<0$ in a unique way in the Hasimoto's framework. 

\subsection{Turbulent features of solutions generating several singularities}  
In the last three decades there has been an intense activity on the analysis of turbulent behaviors of solutions of dispersive equations, as for instance the growth of Sobolev norms since the work of Bourgain in the 90s (\cite{Bourgain1995},\cite{Bourgain1996},\cite{Bourgain1999}).   In our works we have investigated turbulent dynamics through the 1D cubic Schr\"odinger completely integrable equation and its geometric version, the binormal flow, which as we have seen is an equation connected to fluids and superfluids. In the following we shall describe our framework. 

As we have recalled in the previous subsection, an important class of solutions of the binormal flow are the self-similar solutions, that are smooth curves which develop in finite time a singularity in the shape of a corner. 
Making interact several corner singularities is a natural question that has been investigated first by physicists.  Noncircular jets as square jets were studied since the 80s for the turbulent features they produce. For instance experiments were done by Todoya and Hussain (\cite{TodoyaHussain1989}), and numerics by Grinstein and De Vore (\cite{GrinsteinDeVore1996}), see Figure 2. 

\begin{center}
\includegraphics[height=1.2in]{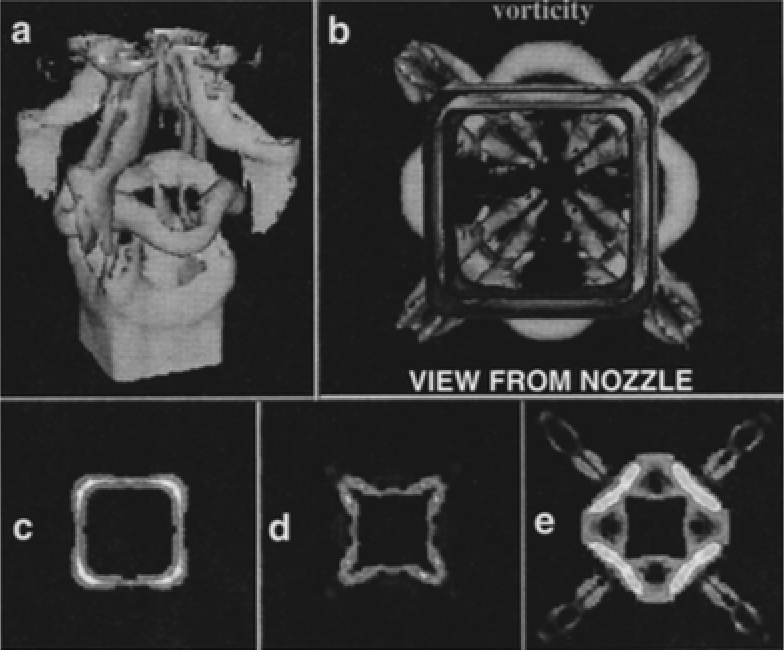}\\
 {\small{Figure 2. Axis switching in numerical simulation of square jets from \cite{GrinsteinDeVore1996}.}}
\end{center}

At the level of the binormal flow this corresponds to consider as initial data a closed curve that is a regular polygon. Such a regular $M-$polygon $\chi_M(0,x)$ with corners located at $x\in\mathbb Z$ (we see here the closed curve as being parametrized, in a periodic way, by $x\in\mathbb R$) is expected to evolve by the binormal flow to skew $Mq-$polygons at times $\frac pq\in\mathbb Q$, as suggested by numerics by Jerrard and Smets (\cite{JerrardSmets2015}), and by De la Hoz and Vega (\cite{DelaHozVega2014}), see Figure 3. This can be seen as a Talbot effect. More precisely, in view of the generation of one corner of angle $\theta_a$ by the self-similar solution $\chi_a$ driven by the filament function $a\frac{e^{i\frac{x^2}{4t}}}{\sqrt{t}}$, and of the relation \eqref{angle}, one expects that the filament function of the $M-$polygon evolution corresponds to considering NLS solutions with initial data of type Dirac comb distribution $\sum_{k\in\mathbb Z}\delta_{\frac kM}$, say instead $\sum_{k\in\mathbb Z}\delta_k$ to simplify the presentation, which is out of reach. However, by supposing Galilean invariance of solutions, De la Hoz and Vega proposed in \cite{DelaHozVega2014} a filament candidate as $u(t)=\phi(t)e^{it\Delta}\sum_{k\in\mathbb Z}\delta_k$, which has an explicit expression at rational times $t=\frac pq$, since the fundamental solution of the periodic Schr\"odinger equation values at rational times:
$$e^{i\frac pq\Delta}\sum_{k\in\mathbb Z}\delta_k=\sum_{j\in\mathbb Z} \tau_j\delta(x-\frac jq),$$
with the coefficients $\tau_j$ given in terms of Gauss sums of exponentials\footnote{Constants as for instance $2\pi$ that are irrelevant for the nature of the phenomena described will be neglected in general all along this article.} (see the computation in \S \ref{section-Talbot}). 
Then, they succeeded to integrate the parallel frame system \eqref{parallel} with filament function $u(\frac pq)$ and proved that the corresponding curves are skew polygons with $Mq$ sides, same angles, and torsion in terms of Gauss sums. This fits with the numerics performed on the square evolution, at the level of binormal flow and of the Schr\"odinger map \eqref{Smap}.

\begin{center}
\includegraphics[height=1.3in]{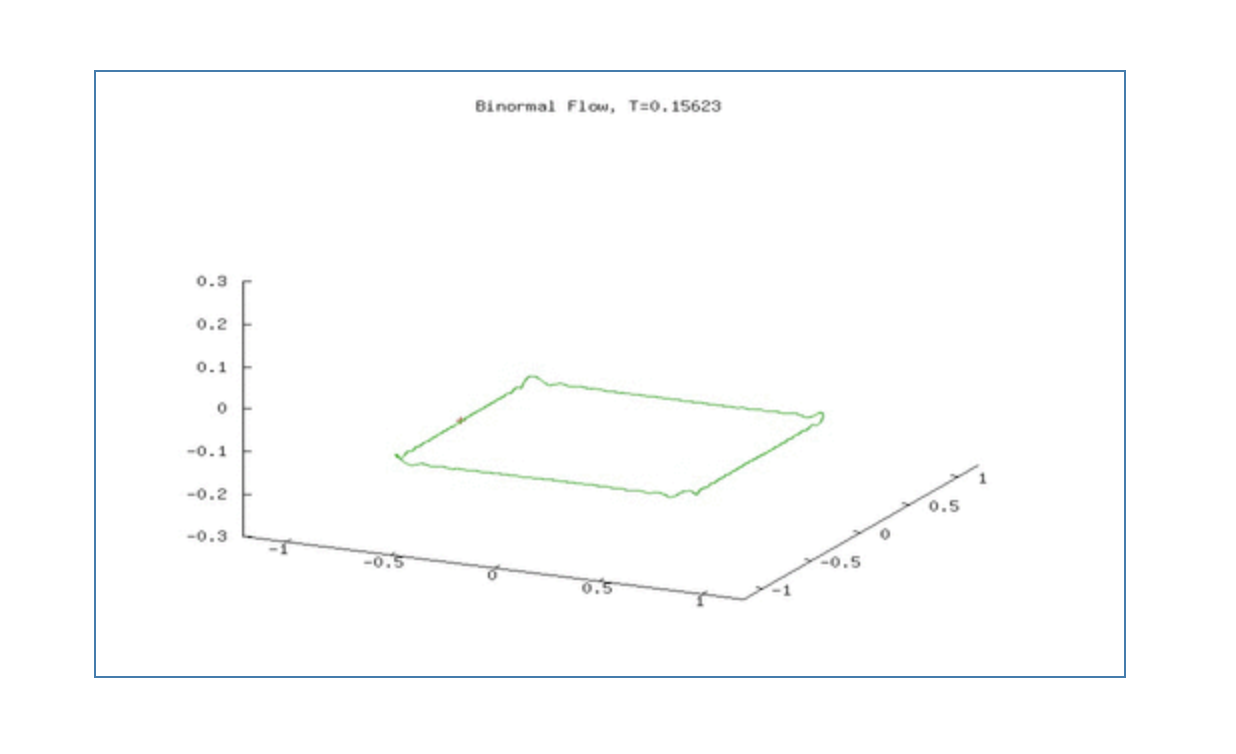}\includegraphics[height=1.3in]{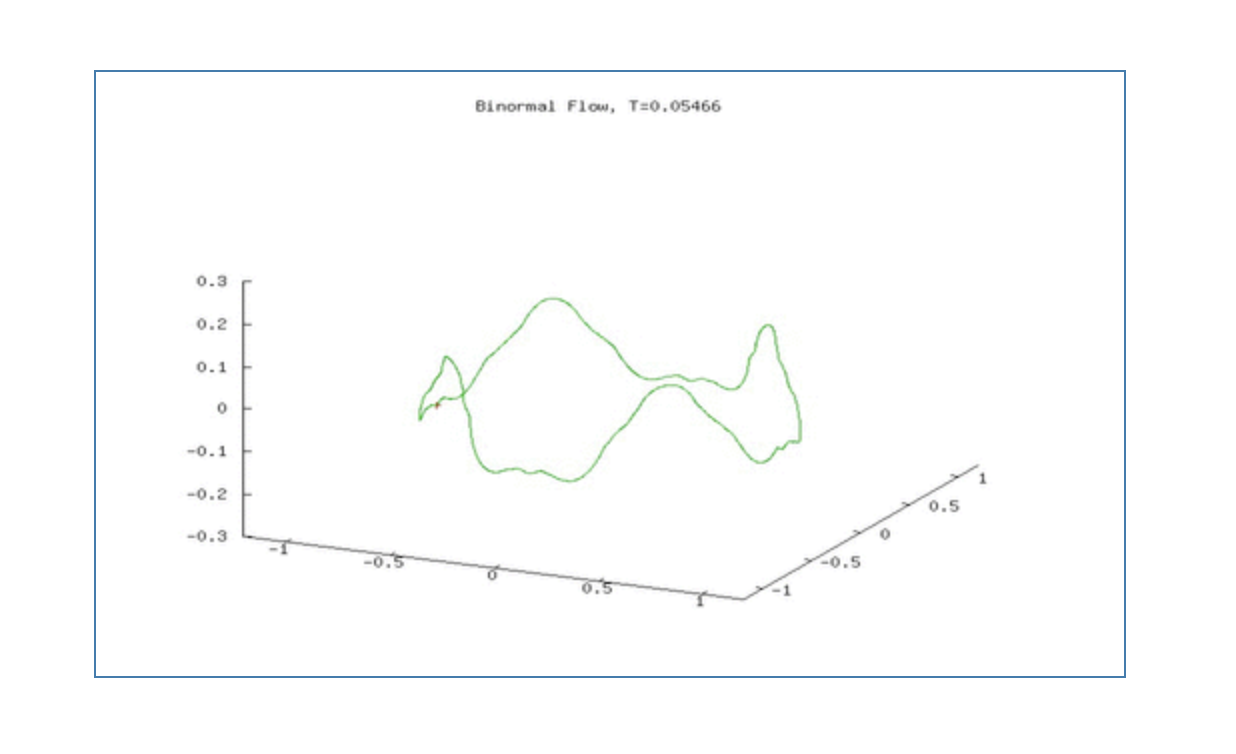}\includegraphics[height=1.3in]{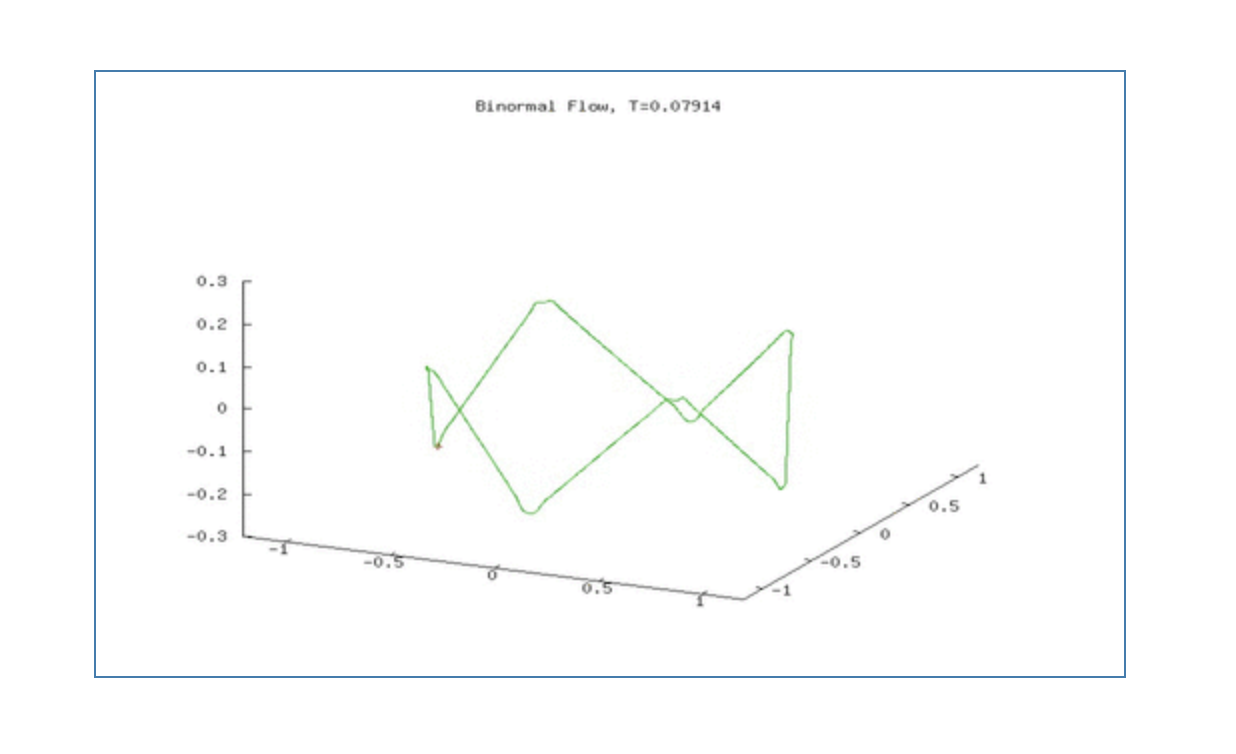}\\
{\small{Figure 3. Numerical binormal flow evolution of a square from \cite{JerrardSmets2015},\\ see \href{https://www.ljll.fr/gallery/html/DSmets\_LIA-fr.htm)}{https://www.ljll.fr/gallery/html/DSmets\_LIA-fr.htm}}}\end{center}

Moreover, other turbulent features were observed numerically, as for instance the growth in phase space at the level of the tangent vectors which model the 
direction of vorticity, and the multifractality of the trajectories. The trajectories of corners $\chi_M(t,0)$ were numerically proved to behave as Riemann's complex function $\sum_{k\in\mathbb Z}\frac{e^{itk^2}-1}{k^2}$ when $M$ tends to infinity (De la Hoz and Vega \cite{DelaHozVega2014}, De la Hoz, Kumar and Vega \cite{delaHozKumarVega2020}, see Figure 4). \medskip

\begin{center}
\includegraphics[height=1.1in]{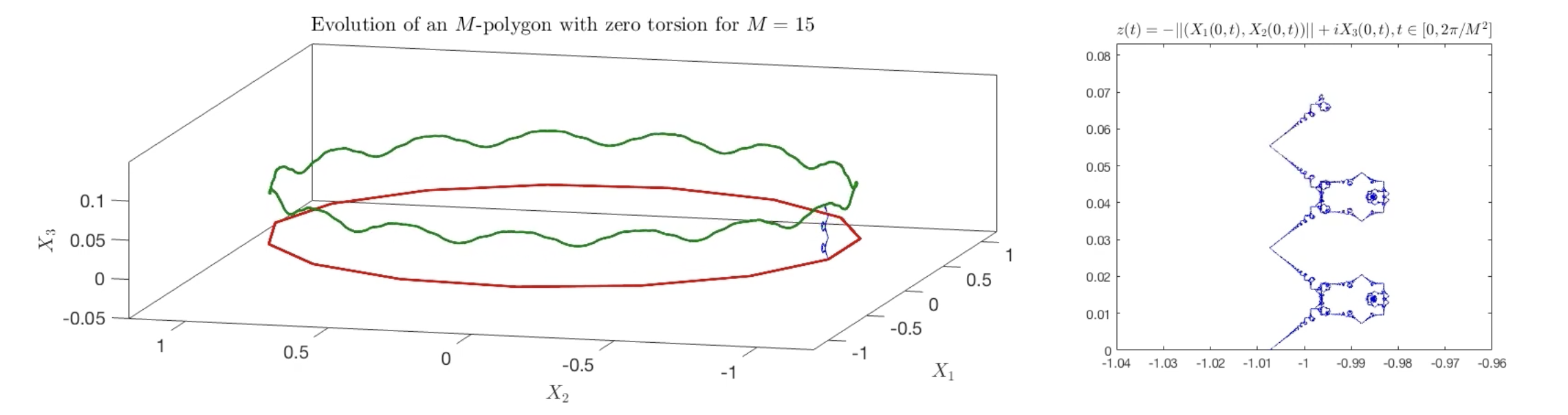}\\
{\small{Figure 4. Left: Evolution by the binormal flow of a M-polygon with $M=15$. Right: the trajectory in time of the solution at $x=0$ (courtesy of Sandeep Kumar, see \href{https://www.youtube.com/watch?v=bwbpKvqGk-o}{https://www.youtube.com/watch?v=bwbpKvqGk-o}).}} 
\end{center}
\medskip

In a new series of papers of the last six years, that we survey here (\cite{BanicaVega2020},\cite{BanicaVega2020bis},\cite{BanicaVega2022},\cite{BanicaVega2022bis},\cite{BanicaVega2024},\cite{BanicaEceizabarrenaNahmodVega2024},\\\cite{BanicaLucaTzvetkovVega2024}), by making interact several corners through the binormal flow, we displayed rigorously a large range of complex behavior as creation of singularities and unique continuation, Fourier growth, Talbot effects, intermittency and multifractality, justifying in particular the previous numerical observations. This has been done by constructing and analyzing binormal flow solutions with initial data of the shape of infinite polygonal lines with only two corners, or made by a regular M-polygon and two half-lines. To do so we constructed a class of well-posedness for the 1D cubic Schr\"odinger equation included in the critical Fourier-Lebesgue space $\mathcal FL^\infty$ and in supercritical Sobolev spaces with respect to scaling. 
\\

{\bf{Acknowledgements:}}  VB is partially supported by the ERC advanced grant GEOEDP and by the French ANR project BOURGEONS. LV is partially supported by MICINN (Spain) CEX2021-001142, PID2021-126813NB-I00 (ERDF A way of making Europe) and IT1247-19 (Gobierno Vasco). The authors are grateful to the referee for the very careful reading of the paper. 

\section{Existence of solutions generating several singularities and unique continuation}\label{section-several}
As we have seen, a corner of angle $\theta_a$ is generated by the binormal flow self-similar solution $\chi_a$, with $a$ related to $\theta_a$ by \eqref{angle}. Moreover, the filament function of $\chi_a$ is $a\frac{e^{i\frac{x^2}{4t}}}{\sqrt{t}}$, solution of the 1D cubic Schr\"odinger equation \eqref{Hasimoto} with $f=\frac{a^2}t$. Due to gauge invariance, as explained in the appendix, the solutions $\chi_a$ can be constructed from the filament function $ae^{ia^2\log t}\frac{e^{i\frac{x^2}{4t}}}{\sqrt{t}}$, a solution of the classical 1D cubic Schr\"odinger equation \eqref{CubicNLS}.
Thus, if we want to construct a binormal flow evolution of a polygonal line with several corners, a first natural step is to look for a solution of \eqref{CubicNLS}, smooth for positive times, with data at $t=0$ related to a superposition of Dirac masses. Then, the second step is to use Hasimoto's approach to construct from this solution a  smooth binormal flow solution for positive times and see if it may have as a limit at $t=0$ a polygonal line. 

Looking for 1D cubic Schr\"odinger solutions on $\mathbb R$ with data of Dirac mass type calls for a review on the state of the art of the Cauchy problem for \eqref{CubicNLS}. 
The rescaling that leaves invariant equation \eqref{CubicNLS} is $\lambda u(\lambda^2 t,\lambda x)$, the invariant space with respect to this scaling on the Sobolev scale is $\dot H^{-\frac 12}$, and the one on the Fourier-Lebesgue scale\footnote{We recall that the Fourier-Lebesgue space $\mathcal F L^p$ is the space of tempered distributions with Fourier transform in $L^p$.} is $\mathcal F L^\infty$. 
The equation is well-posed in $H^s$ for $s\geq 0$ and the flow map is uniformly continuous on bounded sets of $H^s$ (Ginibre and Velo \cite{GinibreVelo1985}, Tsutsumi  \cite{Tsutsumi1987}, Cazenave and Weissler  \cite{CazenaveWeissler1990}). It turns out that this cannot hold in $H^s$ for $s< 0$ (Kenig, Ponce and Vega \cite{KenigPonceVega2001}, Christ, Coliander and Tao \cite{ChristCollianderTao2003}). However, the Sobolev norms of Schwartz solutions have controlled growth for $-\frac 12< s<0$ (Koch and Tataru \cite{KochTataru2018} and Kilip, Visan and Zhang \cite{KillipVisanZhang2018}). Finally, Harrop-Griffths, Killip and Visan proved in 2024 (\cite{Harrop-GriffithKillipVisan2024}) the global well-posedness of \eqref{CubicNLS} in $H^s$ for $s>-1/2$  in the sense that the solution map on the Schwartz class admits a unique continuous extension to $H^s$, $s>-1/2$. Their result is sharp in the sense that for $s<-1/2$  a norm inflation with loss of regularity appear (Kishimoto \cite{Kishimoto2009}, Carles and Kappeler \cite{CarlesKappeler2017}, Oh \cite{Oh2017}), in particular \eqref{CubicNLS} is ill-posed in the Hadamard sense in $H^s$, $s<-1/2$. In what concerns the Fourier-Lebesgue spaces, the problem is known to be locally well-posed on $\mathcal F L^p$ for $p<+\infty$  (Vargas and Vega \cite{VargasVega2001}, Gr\"unrock  \cite{Grunrock2005}). In summary, we presently have a well-posedness theory that misses the critical spaces with respect to scaling. 

Let us notice that the Dirac mass and its linear Schr\"odinger evolution, the fundamental Schr\"odinger solution, are of borderline regularity $H^{-\frac12-\epsilon}$ for all $\epsilon>0$ and $\mathcal F L^\infty$, and therefore do not fit in the scope of applicability of the previously described well-posedness results.
Moreover, there are simple explicit solutions of the cubic Schr\"odinger equation \eqref{CubicNLS} for $t > 0$, having the same kind of borderline regularity, given by
$$u_\alpha(t, x) =\alpha e^{i|\alpha|^2\log t} \,\,e^{it\Delta}\delta_0= \alpha e^{i|\alpha|^2\log t} \,\,\frac{e^{i\frac{x^2}{4t}}}{\sqrt t},\quad\alpha\in\mathbb C.$$
Actually these solutions were displayed in the proof of Kenig, Ponce and Vega in \cite{KenigPonceVega2001} to the non-uniqueness of the Cauchy problem of \eqref{CubicNLS} with initial data a Dirac mass. Also, the self-similar solutions of the binormal flow are obtained by the Hasimoto approach also from them, as they differ from $|\alpha|\frac{e^{i\frac{x^2}{4t}}}{\sqrt{t}}$ just by a space-independent phase. This type of solutions allowed recently for a new insight on \eqref{CubicNLS} at critical regularity. Also solutions which are smooth perturbations of $u_\alpha$ were constructed in our previous articles \cite{BanicaVega2009},\cite{BanicaVega2012},\cite{BanicaVega2013},\cite{BanicaVega2015}, see also \cite{Guerin2023}. Then, following an ansatz of Kita proposed in \cite{Kita2006} for subcritical nonlinearities (subcubic), we constructed for \eqref{CubicNLS} solutions of type superpositions of $u_\alpha$:
\begin{equation}\label{superp}
\sum_{j\in\mathbb Z} A_k(t)e^{it\Delta}\delta_k=\sum_{k\in\mathbb Z} A_k(t)\frac{e^{i\frac{(x-k)^2}{4t}}}{\sqrt t},
\end{equation}
as follows.
\begin{theorem}\label{thNLScrit} (Construction of critical NLS solutions, \cite{BanicaVega2022}) Let $s>\frac 12,0<\gamma<1$, $\{\alpha_k\}\in l^{2,s}$, i.e. $\|\{\alpha_k\}\|_{l^{2,s}}^2:=\sum_k|\alpha_k|^2(1+k^2)^s<\infty$. 
 There exists $T>0$ and a unique solution of \eqref{CubicNLS} on $(0,T)$ of the form 
\begin{equation}\label{superpsol}
u_{\{\alpha_k\}}(t,x)=\sum_{k\in\mathbb Z}A_k(t)e^{it\Delta}\delta_k(x),
\end{equation}
 with 
 $$A_k(t)=e^{- i(|\alpha_k|^2-2\sum_{j\in\mathbb Z}|\alpha_j|^2)\log t}(\alpha_k+R_k(t)),$$
 and
$$\sup_{0<t<T}t^{-\gamma}\|\{R_k(t)\}\|_{l^{2,s}}+t\,\|\{\partial_t R_k(t)\}\|_{l^{2,s}}<C.$$
Moreover, if $s\geq 1$ then the solution can be extended to $(0,\infty)$. 
\end{theorem}

The $l^{2,s}$ hypothesis was relaxed to $l^{p,s}$ by Bravin and Vega (\cite{BravinVega2022}). Also, solutions that are smooth perturbations of $u_{\{\alpha_k\}}$ were constructed by Gu\'erin (\cite{Guerin2022}). 

Let us note that by using the pseudo-conformal transformation 
$$
v(\tau,y):=\frac{e^{i\frac{y^2}{4\tau}}}{\sqrt {\tau}}\,\overline u(\frac 1\tau,\frac y{\tau}),
$$
equation \eqref{CubicNLS} is transformed into 
$$
iv_t+v_{xx}+ \frac 1t|v|^2v=0.
$$
Moreover, the  ansatz \eqref{superp} translates into simply being in the periodic setting of this equation. 
In particular, for this periodic 1D cubic NLS with time-variable coefficient we proved by Theorem \ref{thNLScrit} the existence of wave operators. We proved also asymptotic completeness in \cite{BanicaLucaTzvetkovVega2024}, that we review in \S \ref{section-blup}. 

The proof goes as follows: plugging the ansatz $u(t,x)=\sum_{k\in\mathbb Z}A_k(t)e^{it\Delta}\delta_k(x)$ into equation \eqref{CubicNLS} leads to a discrete nonautonomous Hamiltonian system on $\{A_k(t)\}$. We solve this discrete system by a fixed point argument, based on integrations by parts from the nonresonant phases. 
In \S \ref{subsectproofNLS} we give more details on the proof.

The NLS solutions in Theorem \ref{thNLScrit} are the starting point, via the Hasimoto transform, that allowed us to construct the evolution of polygonal lines by the binormal flow as follows.

\begin{theorem}\label{thevpoly}(Evolution of polygonal lines by the binormal flow, \cite{BanicaVega2020}) 
We consider an arclength parametrized polygonal line $\chi_0$ with corners located at $x=k\in\mathbb Z$, of angles\footnote{i.e. $\theta_k$ is the angle of the self-similar binormal flow solution $\chi_{a_k}$, see \eqref{angle}.}\footnote{The fact that $\{a_k\}\in l^{2,3}$ implies that the angles $\theta_j$ tend to $\pi$ as $|j|\rightarrow\infty$.}  $\theta_k$ such that 
$a_k:=\sqrt{-\frac {2}{\pi}\log\left(\sin\left(\frac{\theta_k}{2}\right)\right)}\in l^{2,\frac 32^+}$.
Then, there exists $\chi(t)$ smooth solution of the binormal flow on $\mathbb R^*$, weak solution on $\mathbb R$, with
$$|\chi(t,x)-\chi_0(x)|\leq C\sqrt{t},\quad \forall x\in\mathbb R,|t|\leq 1.$$
The solution is unique in the framework of curves having as filament functions the solutions in Theorem \ref{thNLScrit}.
\end{theorem} 
We note that the NLS solutions in Theorem \ref{thNLScrit} blow-up at time $t=0$ with a loss of phase. Phase blow up phenomena were encountered for the Schr\"odinger equation since the works of Merle in the 90s (\cite{Merle1992}, see also \cite{MerleRaphaelSzeftel2013}). Here, we see that despite this loss of phase the solutions of the binormal flow associated via Hasimoto's transform are uniquely continued after the singularity time $t=0$. This can be seen as a way of continuation for NLS despite the phase loss. In \S \ref{subsectproofBF} we give details on the proof of Theorem \ref{thevpoly}. 

\subsubsection{Sketch of the proof of Theorem \ref{thNLScrit}}\label{subsectproofNLS}
By plugging the ansatz $u(t)=\sum_{k\in\mathbb Z}A_k(t)e^{it\Delta}\delta_k$ into equation \eqref{CubicNLS}
we obtain:
$$\sum_{k\in\mathbb Z}i\partial_t A_k(t)e^{it\Delta}\delta_k=|\sum_{j\in\mathbb Z}A_j(t)e^{it\Delta}\delta_j|^2(\sum_{j\in\mathbb Z}A_j(t)e^{it\Delta}\delta_j).$$
The family $e^{it\Delta}\delta_k(x)=\frac{e^{i\frac{(x-k)^2}{4t}}}{\sqrt{t}}$ is, modulo some constants, an orthonormal family of $L^2(0,4t)$, so
 by taking the scalar product with $e^{it\Delta}\delta_k$ we obtain the following nonautonomous Hamiltonian system of the coefficients $\{A_k(t)\}$:
$$i\partial_t A_k(t)=\int_0^{4 t} \mathcal |\sum_{j\in\mathbb Z}A_j(t)\frac{e^{i\frac{(x-j)^2}{4t}}}{\sqrt{t}}|^2(\sum_{j\in\mathbb Z}A_j(t)\frac{e^{i\frac{(x-j)^2}{4t}}}{\sqrt{t}})\frac{-e^{i\frac{(x-k)^2}{4t}}}{\sqrt{t}}\,dx.$$
$$=\frac{1}{4t}\sum_{k-j_1+j_2-j_3=0}e^{-i\frac{k^2-j_1^2+j_2^2-j_3^3}{4t}}A_{j_1}(t)\overline{A_{j_2}(t)}A_{j_3}(t).$$
We split the summation indices into the following two sets:
$$NR_k=\{(j_1,j_2,j_3)\in\mathbb Z^3, k-j_1+j_2-j_3=0, k^2-j_1^2+j_2^2-j_3^2\neq0\},$$
$$Res_k=\{(j_1,j_2,j_3)\in\mathbb Z^3, k-j_1+j_2-j_3=0, k^2-j_1^2+j_2^2-j_3^2=0\}.$$\smallskip
As we are in one dimension, the resonant set is simply:
$$Res_k=\{(k,j,j), (j,j,k), j\in\mathbb Z\},$$
so the system writes:
\begin{equation}\label{systAj}
i\partial_t A_k(t)=\frac 1{4t}\sum_{(j_1,j_2,j_3)\in NR_k}e^{-i\frac{k^2-j_1^2+j_2^2-j_3^3}{4t}}A_{j_1}(t)\overline{A_{j_2}(t)}A_{j_3}(t)+\frac{A_k(t)}{4 t}(2\sum_j|A_j(t)|^2-|A_k(t)|^2).
\end{equation}
The system preserves the ``mass" $\sum_k|A_k(t)|^2$, since for any $a:\mathbb Z\rightarrow\mathbb R$ we can compute, by using a standard symmetrization argument and the fact that $a(k)-a(j_1)+a( j_2)-a(j_3)$ vanishes on the resonant set:
$$\partial_t\sum_k a(k)|A_k(t)|^2=\frac 1{4t}\sum_{k,NR_k}(a(k)-a(j_1)+a(j_2)-a(j_3))e^{-i\frac{k^2-j_1^2+j_2^2-j_3^3}{4t}}A_{j_1}(t)\overline{A_{j_2}(t)}A_{j_3}(t)).$$

We shall obtain the existence of solutions of system \eqref{systAj} of the form 
$$A_k(t)=e^{- i(|\alpha_k|^2-2M)\log t}(\alpha_k+R_k(t)),$$ 
with $R_k$ in
\begin{equation}\label{Xgamma}
X^\gamma:=\{\{f_k\}\in\mathcal C([0,T];l^{2,s}), \sup_{0\leq t<T}t^{-\gamma}\|\{f_k(t)\}\|_{l^{2,s}}+t\|\{\partial_t f_k(t)\}\|_{l^{2,s}}<\infty\}.
\end{equation}
By using the ``mass" $\sum_j|A_j(t)|^2$ conservation,  and the evolution law of $|A_k(t)|^2$, we reduce to proving that  the operator $\Phi$ defined on the space $\mathcal C([0,T];l^{2,s})$ as

{\small{$$\Phi_k(\{R_j\})(t)=i\int_0^t g_k(\tau)d\tau-i\int_0^t \int_0^\tau \Im (g_k(s)\overline{(\alpha_{k}+R_{k}(s)})ds\,(\alpha_{k}+R_{k}(\tau))\frac{d\tau}{4\pi\tau},$$}}
where
$$g_k(t)=\,\frac{1}{4\pi t}\sum_{(j_1,j_2,j_3)\in NR_k}e^{-i\frac{k^2-j_1^2+j_2^2-j_3^2}{4t}}e^{-i\frac{|\alpha_k|^2-|\alpha_{j_1}|^2+|\alpha_{j_2}|^2-|\alpha_{j_3}|^2}{4\pi}\log t}\times$$
$$\times (\alpha_{j_1}+R_{j_1}(t))\overline{(\alpha_{j_2}+R_{j_2}(t))}(\alpha_{j_3}+R_{j_3}(t)),$$
is a contraction on a small ball of the space $X^\gamma$ defined in \eqref{Xgamma}. The estimates in the fixed point arguments are based on integrations by parts from the nonresonant phases, which yield decay in time necessary for integration near zero.

\subsubsection{Sketch of the proof of Theorem \ref{thevpoly}}\label{subsectproofBF}

The proof can be synthethized as follows:

\medskip
\noindent
{\bf{Step 1:}} define $\alpha_k=a_ke^{i\gamma_k}$, with $a_k,\gamma_k$ expressed in terms of the curvature angles and torsion angles of the polygonal line $\chi_0$ in a way to be specified later (for instance $a_k$ is determined from the curvature angle $\theta_k$ by formula \eqref{angle}),

\medskip
\noindent
{\bf{Step 2:}}  use Theorem \ref{thNLScrit} to get a solution $u_{\{\alpha_k\}}$ of the NLS equation \eqref{CubicNLS}, smooth for $t>0$,

\medskip
\noindent
{\bf{Step 3:}}  construct for $t>0$ the binormal flow solution $\chi$ obtained from $u_{\{\alpha_k\}}$ by the Hasimoto construction detailed in \S\ref{sectHas}. In particular, the frames $(T,e_1,e_2)$ satisfy the ODEs in time and in space \eqref{ODEx}-\eqref{ODEt} with system matrices involving $u_{\{\alpha_k\}}$,

\medskip
\noindent
{\bf{Step 4:}}  get a trace $\chi(0)$ for $\chi(t)$ as $t$ goes to zero with the rate of convergence $\sqrt t$, as from \eqref{ev} we have 
$$|\partial_t\chi(t)|\leq 2|u_{\{\alpha_k\}}(t)|\leq \frac{C(\|\alpha_j\|_{l^1})}{\sqrt t}.$$ 
The goal is now to show that $\chi(0)$ and $\chi_0$ coincide,

\medskip
\noindent
{\bf{Step 5:}}  for $x\notin\mathbb Z$ obtain a limit as $t$ goes to zero for $T(t,x)$, based on estimating oscilatory integrals involving $u_{\{\alpha_k\}}$,

\medskip
\noindent
{\bf{Step 6:}}  prove that the vectors $T(0,x)$ are constant for $k<x<k+1$, again based on estimating oscilatory integrals involving $u_{\{\alpha_k\}}$, so $\chi(0)$ is a polygonal line,

\medskip
\noindent
{\bf{Step 7:}}  recover a self-similar binormal flow profile $\chi_{a_k}$ on self-similar paths towards $(0,k)$: first conclude by Arzela-Ascoli's theorem that there is a limit for the frames via self-similar path $(t_n^k,k+x\sqrt{t_n^k})$ for some sequence $t_n^k\overset{n\rightarrow\infty}{\rightarrow} 0$, and that the limit frame satisfies the EDO of the self-similar profile $\chi_{a_k}$. Then conclude by uniqueness that there exists a rotation $\Theta_k$ such that
\begin{center}$\underset{n\rightarrow\infty}{\lim}T(t_n^k,k+x\sqrt{t_n^k})=\Theta_k (T_{a_k}(x)),$\end{center}
where $T_{a_k}$ is the tangent vector of the profile curve $\chi_{a_k}(1)$, which has a limit $A^\pm_{a_k}$ at $\pm\infty$. 
This allows for recovering the curvature angle of $\chi_0$ at $x=k$, since 
the values $T(0,k^\pm)$ are limits of $T(0,k+x\sqrt{t_n^k})$, which in turn by Step 5 can be approximated by $T(t_n^k,k+x\sqrt{t_n^k})$ which eventually is close to $\Theta_kA^\pm_{a_k},$ by taking $x$ large enough,

\medskip
\noindent
{\bf{Step 8:}}  recover the torsion angles of $\chi_0$ by using also a similar analysis for the modulated normal vectors 
$$\tilde N(t,x)=e^{i\sum_j|\alpha_j|^2\log\frac{x-j}{\sqrt{t}}}N(t,x).$$ At this stage the relation between the sequences $\{\alpha_k\}$ and $\{\gamma_k\}$ and the curvature and torsion angles of the initial curve $\chi_0$ becomes explicit,

\medskip
\noindent
{\bf{Step 9:}}  unique continuation after singularity time: since now that the theorem is proved for positive time evolutions, we use the time invariance of binormal flow to construct $\chi$ for negative times by $\chi(t,x)=\tilde\chi(-t,-x)$, where $\tilde\chi$ is the evolution for positive times with initial data the polygonal line $\tilde\chi(0,x)=\chi(0,-x)$. Thus the  continuation of $\chi$ for negative times is done by the evolution of $\chi(0)$ with the sense of parametrization inverted. 

Let us recall that in the simple case of self-similar solutions $\chi_a$, for which $\chi_a(0)$ is a curve with one corner of angle $\theta_a$, for negative times the solution is a rotation and symmetry of the solution for positive times. This is due in particular to the fact that the filament function is the same for the evolution of $\chi_a(0,x)$ as well as for the evolution of $\chi_a(0,-x)$, namely $a\frac{e^{i\frac{x^2}{4t}}}{\sqrt t}$.  In our case the inversion of the sense of parametrization has an evolution that is not geometrically trivially related to the evolution of $\chi(0,x)$ for positive times. This can be seen from the fact that the filament function used for constructing the evolution of $\chi(0,x)$ is 
$$\sum_ke^{- i(|\alpha_k|^2-2\sum_{j\in\mathbb Z}|\alpha_j|^2)\log t}(\alpha_{k}+R_{k}(t))\frac{e^{i\frac{(x-k)^2}{4t}}}{\sqrt t},$$ 
while the filament function used for constructing the evolution of $\chi(0,-x)$ is 
$$\sum_ke^{- i(|\alpha_{-k}|^2-2\sum_{j\in\mathbb Z}|\alpha_j|^2)\log t}(\alpha_{-k}+\tilde R_{k}(t))\frac{e^{i\frac{(x-k)^2}{4t}}}{\sqrt t}.$$ 
As a byproduct this is a way to uniquely continuate the NLS solutions after a generation of singularites with phase loss.

\section{Fourier growth}
As said before, in the last decades there has been an intense activity on the analysis of turbulent dynamics related to dispersive equations, as for instance the growth of Sobolev norms since the work of Bourgain in the 90s. To quote just a few, such studies were done for non-integrable equations as the linear Schr\"odinger equation with potential, the 2D cubic NLS, systems of 1D cubic NLS (Bourgain \cite{Bourgain1995},\cite{Bourgain1996},\cite{Bourgain1999}, Kuksin \cite{Kuksin1996},\cite{Kuksin1997}, Staffilani \cite{Staffilani1997}, Colliander, Keel, Staffilani, Takaoka and Tao \cite{CollianderKeelStaffilaniTakaokaTao2010}, Sohinger \cite{Sohinger2011}, Carles and Faou \cite{CarlesFaou2012}, Gr\'ebert, Paturel and Thomann \cite{GrebertPaturelThomann2013}, Delort \cite{Delort2014}, Hani \cite{Hani2014}, Hani, Pausader, Tzvetkov and Visciglia \cite{HaniPausaderTzvetkovVisciglia2015}, Guardia and Kaloshin \cite{GuardiaKaloshin2015}, Planchon, Tzvetkov and Visciglia \cite{PlanchonTzvetkovVisciglia2017}, Bambusi, Gr\'ebert, Maspero and Robert \cite{BambusiGrebertMasperoRobert2018}, Carles and Gallagher \cite{CarlesGallagher2018}, Deng and Germain \cite{DengGermain2019}, Thomann \cite{Thomann2021}, Bambusi, Langella and Montalto  \cite{BambusiLangellaMontalto2022},  Giuliani and Guardia \cite{GiulianiGuardia2022}, Faou and Rapha\"el \cite{FaouRaphael2023}, Guardia, Hani, Haus, Maspero and Procesi \cite{GuardiaHaniHausMasperoProcesi2023}, ...)
In general the 1D cubic NLS and the Schr\"odinger map were left aside because of their complete integrability. We mention at this point that in the periodic defocusing case the $H^s$ norms with $n\geq 1$ were proved to remain bounded for all times by Kappeler, Schaad and Topalov (\cite{KappelerSchaadTopalov2017}). We note however that turbulent behavior was observed for abstract integrable equations as of Szeg\"o type (G\'erard and Grellier \cite{GerardGrellier2012},\cite{GerardGrellier2017}, Pocovnicu  \cite{Pocovnicu2011}, G\'erard, Lenzmann, Pocovnicu and Rapha\"el \cite{GerardLenzmannPocovnicuRaphael2018}, Biasi and Evnin \cite{BiasiEvnin2022}, G\'erard, Grellier and He \cite{GerardGrellierHe2022}, G\'erard and Lenzmann \cite{GerardLenzmann2024},...). 

In \cite{BanicaVega2020bis}-\cite{BanicaVega2022} we displayed a Fourier modes amplitude growth in time, observed in a frequency region that shifts in time, for the Schr\"odinger map \eqref{Smap}, 
that on one hand is directly connected via Hasimoto's transform with the 1D cubic NLS, and on the other hand is also directly connected to the binormal flow, which arises as a model for vortex filaments, as explained before.

More precisely, let us first recall that the Schr\"odinger map \eqref{Smap} has interaction energy: 
$$\int|T_x(t,x)|^2dx.$$
From now on we consider the binormal flow solutions of Theorem \ref{thevpoly} that are constructed from the NLS solutions $u_{\{\alpha_j\}}$ in Theorem \ref{thNLScrit}. Their tangent vector $T$ satisfies the Schr\"odinger map \eqref{Smap} strongly on $(-1,1)\setminus\{0\}$ and weakly on $(-1,1)$. These solutions  have infinite interaction energy since, in view of \eqref{ev}, $|T_x(t,x)|=|u_{\{\alpha_k\}}(t,x)|^2$ is a periodic function. However, we shall prove they have a new type of finite energy. 

We recall again that the Constantin-Fefferman-Majda blow-up criterium (\cite{ConstantinFeffermanMajda1996}) on the variation of the direction of vorticity, writes  $\int_0^t\|\nabla(\frac{\omega}{|\omega|})(\tau)\|_{L^\infty_x}^2d\tau=\infty$. Here, having in mind that the tangent vector $T$ models the direction of vorticity, see also evidence in this sense in Theorem 1 in \cite{FontelosVega2023}, we have:
$$\int_0^t\|T_x(\tau)\|_{L^\infty}^2d\tau=\int_0^t\|u_{\{\alpha_k\}}(\tau)\|_{L^\infty}^2d\tau=\int_0^t \frac {\| \sum_k((\alpha_k+R_k(\tau))e^{i\frac{(x-k)^2}{4\tau}})\|_{L^\infty}^2}{\tau}d\tau,$$
which is infinite for generic sequences $\{\alpha_j\}$ since $R_k(\tau)$ behaves as $\tau^{1^-}$ uniformly in $k$ as $\tau$ goes to zero. In \cite{BanicaVega2020bis}-\cite{BanicaVega2022} we analyzed $T_x$ in the phase variable. As a first result we identified a finite energy density framework which points out a growth at large frequencies as follows.

\begin{theorem}\label{thenergy}(A finite energy density framework, \cite{BanicaVega2020bis}) Let $T$ be the solution of \eqref{Smap} obtained as tangent vector of the binormal flow evolution of a polygonal line from Theorem \ref{thevpoly}. Then 
$$\Xi(t):=\underset{n\rightarrow\infty}{\lim}\int_n^{n+1}|\widehat{T_x}(t,\xi)|^2d\xi$$
is conserved for $t\in(0,1)$ with a discontinuity at $t=0$.
Moreover, there is an instantaneous growth for positive times at large frequencies:
$$\forall n, \quad\Xi(0)= \int_n^{n+1}|\widehat{T_x}(0,\xi)|^2d\xi=4\sum_k(1-e^{-\pi \alpha_k^2})< 4\pi\sum_k\alpha_k^2=\Xi(t).$$
\end{theorem}
In particular we can see $|\widehat{T_x}(t,\xi)|^2$ as an asymptotic energy density in phase space. Moreover, we observed an energy cascade in the following precise sense, even for solutions of \eqref{Smap} obtained as tangent vector of the binormal flow evolution of a polygonal line with only two corners from Theorem \ref{thevpoly}.
 
\begin{theorem}\label{thgrowth}{(Energy density growth, \cite{BanicaVega2022})}
Let $T_0:\mathbb R\rightarrow \mathbb S^2$, constant on $(-\infty,-1),(-1,1)$ and $(1,+\infty)$, with jumps of same angle $\theta\in (0,\pi)$ at $x\in\{-1,1\}$. Then there exists a solution $T$ of the Schr\"odinger map \eqref{Smap} on $t\in (0,1)$, with $T_0$ as trace at time $t=0$ in the sense that $T_0(x)=\lim_{t\rightarrow 0}T(t,x)$ for all $x\in\mathbb R$, satisfying:\\
i) there exists $C_\theta>0$ such that:
$$\sup_{\xi\in B( \pm \frac{1}{ t},\sqrt{t})}|\widehat{T_x}(t,\xi)|= C_\theta |\log t|.$$
ii) for $\xi\notin B( \frac{1}{t},\frac{3}{4 t})\cup B(- \frac{1}{t},  \frac{3}{4t})$ we have an upper-bound of $|\widehat{T_x}(t,\xi)|$ depending only on $\theta$.
\end{theorem}

The growth is valid also for polygonal lines with many corners. This results confirms the numeric results obtained by de la Hoz and Vega in \cite{DelaHozVega2018} for periodic piecewise constant data, in the case when  $T_0$ is the tangent vector of a regular polygon. Last but not least, let us underline that the growth of $T_x$ is in terms of the Fourier-Lebesgue norm $\mathcal FL^\infty$, which is critical with respect to scaling for the 1D cubic NLS equation \eqref{CubicNLS} to which \eqref{Smap} is linked via Hasimoto's method.

Let us also note that in view of equations \eqref{ev} satisfied by $T$ and of the particular oscillatory ansatz of $u$ in Theorem \ref{thNLScrit} from which $T$ is constructed, we have:
$$T_t=-\frac x{2t}\Re(\overline{u}\,N) + \text{ Remainder term},\quad T_x=\Re(\overline u\, N).$$
Passing in Fourier gives:
$$
\partial_t\widehat{T_x}=\frac{\xi}{2t}\partial_\xi \widehat{T_x}+ \text{ Remainder term}.
$$
This phenomena is reminiscent of the works of 
Apolin\'ario et al. \cite{ApolinarioBeckChevillardGallagherGrande2023}-\cite{ApolinarioChevillardMourrat2022}
where they proposed abstract linear equations that mimic the phenomenology of energy cascades when the external force is a statistically homogeneous and stationary stochastic process. 
Indeed, these equations have a transport rate in wavenumber space, 
independent of time and frequency (see (2.1) in \cite{ApolinarioBeckChevillardGallagherGrande2023}).
Existence of energy cascades for linear systems have been also proved by Colin de Verdi\`ere and Saint-Raymond in \cite{deVerdiereSaintRaymond2020}.

\subsubsection{Sketch of the proof of Theorem \ref{thgrowth}}
We consider $\chi_0$ a polygonal line with two corners, having $T_0$ as tangent vector. We can thus apply Theorem \ref{thevpoly} to get  a binormal flow solution $\chi$ with data $\chi_0$, and a Schr\"odinger map solution $T=\partial_x\chi$ with initial data $T_0$. From the construction in the proof of Theorem \ref{thevpoly} we also get that $T$ can be completed to a frame $(T,e_1,e_2)$ such that $T$ and $N=e_1+ie_2$ satisfy \eqref{ev} for some solution $u_{\{\alpha_j\}}$ of type \eqref{superpsol}. We aim to show a $\log t$ growth of $\widehat{T_x}(t,\xi)$ for frequencies $\xi\in B(\frac 1t,\sqrt{t})$. We have, by using $T_x=\Re(\overline{u}N)$ from \eqref{ev}:
$$\widehat{T_x}(t,\xi)=\int e^{ix\xi}\,\Re(\sum_k\overline{A_k(t)}\frac{e^{-i\frac{(x-k)^2}{4t}}}{\sqrt{t}} N(t,x))dx$$
$$=\frac{e^{it\xi^2}}{2\sqrt{t}}\sum_{k}e^{i k\xi}\,\overline{A_k(t)}\int e^{-i\frac{(x-k-2 t\xi)^2}{4t}}N(t,x)dx+ \mbox{ a similar term.}$$
We remove $ B(k+2 t\xi,\sqrt{t})$ from the domain of integration, as its contribution is of order $1$. Then on the remaining domain we use integrations by parts and the expression $N_x=-uT$ from \eqref{ev} to reduce to:
$$ i\sum_{k,j}\,\overline{A_k(t)}A_j(t) e^{-i\frac{k^2-j^2}{4t}}\int e^{i\frac{x(k-j+2 t\xi)}{2t}}\frac{1_{|_{x-k-2 t\xi>\sqrt{t}}}}{x-k-2t\xi}\,T(t,x)\,dx.$$
If $|k- j+ 2 t\xi|>\sqrt{t}$ the corresponding term is proved to be bounded also by using several integrations by parts involving \eqref{ev}. We thus get the general estimate, that indicate where to look for a potential growth:
$$\left|\widehat{T_x}(t,\xi)-i\sum_{|k- j+ 2 t\xi|<\sqrt{t}}\overline{\alpha_k+R_k(t)}(\alpha_j+R_j(t))\,z\right.$$
$$\left.\times\int e^{i\frac{x( k- j+ 2 t\xi)}{2t}}\left(\frac{1_{|_{x-k-2 t\xi>\sqrt{t}}}}{x-k-2 t\xi}-\frac{1_{|_{x-j+2 t\xi>\sqrt{t}}}}{x-j+ 2 t\xi}\right)\,T(t,x)\,dx\right|\leq C,$$
where 
$$z=e^{-i(|\alpha_k|^2-|\alpha_j|^2)\log t}e^{-i\frac{k^2-j^2}{4t}}.$$
Considering $\xi\in B(\frac{1}{t}, \sqrt{t})$ implies $2 t\xi\in B(2,2t\sqrt{t})$, so the summation condition implies:
$$|k- j+ 2t \xi|<\sqrt{t}\Longrightarrow j=k+2.$$
Thus, restricting to polygonal lines with finite number of corners we end up with an estimate that writes for the two corners case (when $\alpha_j=0$ for $|j|\geq 2$), for $\xi\in B(\frac{1}{t}, \sqrt{t})$:
$$
\left|\widehat{T_x}(t,\xi)-i\overline{\alpha_{-1}}\alpha_{1}\,z \int\left(\frac{1_{|_{x-1>\sqrt{t}}}}{x-1}-\frac{1_{|_{x+1>\sqrt{t}}}}{x+1}\right)T(t,x)dx\right|\leq C.$$
In particular we have
$$|\widehat{T_x}(t,\xi)-i\overline{\alpha_{-1}}\alpha_{1}\,z\, \int_{\sqrt{t}<|x-1|<\frac 13\}\cup\{\sqrt{t}<|x+1|<\frac 13\}}(\frac{1}{x-1}-\frac{1}{x+1})\,T(t,x)\,dx|\leq C.$$
Next we use the following result on the convergence of the tangent vector from \cite{BanicaVega2020}:
$$|T(t,x)-T(0,x)|\leq C(\|\{\alpha_j\}\|_{l^{1,1}})(1+|x|)\sqrt{t}\left(\frac 1{d(x,\frac 12\mathbb Z)}+\frac 1{d(x,\mathbb Z)}\right),$$
to get
$$|\widehat{T_x}(t,\xi)-i\overline{\alpha_{-1}}\alpha_{1}\,z\,\int_{\{\sqrt{t}<|x-1|<\frac 13\}\cup\{\sqrt{t}<|x+1|<\frac 13\}}(\frac{1}{x-1}-\frac{1}{x+1})\,T(0,x)\,dx|\leq C.$$
As $T(0,x)$ is piecewise constant direct integration yields a $\log t$ growth.\\

\section{Talbot effects} \label{section-Talbot}
In 1836 Talbot, inventor of photography independently of Daguerre, observed a diffraction effect of light. By illuminating a grating that has equally spaced transparent and opaque slits, the pattern of the grating can be observed away from the grating at a distance called nowadays Talbot distance. Also, at rational fractions $\frac pq$ of the Talbot distance the image observed consists of $q$ overlapping copies of the pattern, see Figure 5. Then the Talbot effect has been forgotten, rediscovered by Lord Rayleigh in 1881 and forgotten again. Nowadays it is well-known and also known to be related to a phenomenon in quantum physics called quantum revivals, that concerns reconstruction after a finite time of a wave packet, see Figure 5. For more details on this topic see the expository article \cite{BerryMarzoliSchleich2001} of Berry, Marzoli and Schleich.

\begin{center}
$$\includegraphics[width=1.5in]{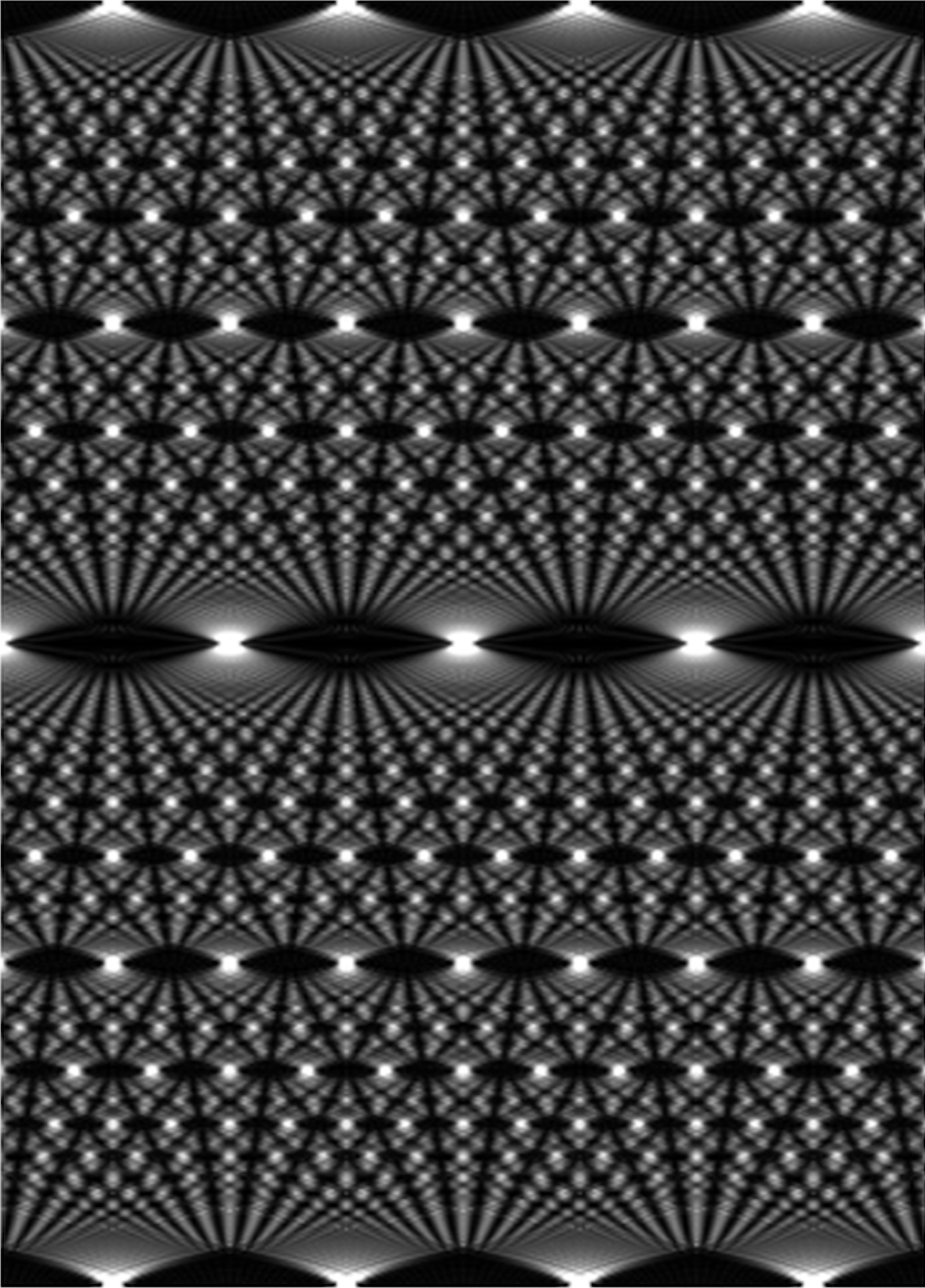}\quad \includegraphics[width=2.1in]{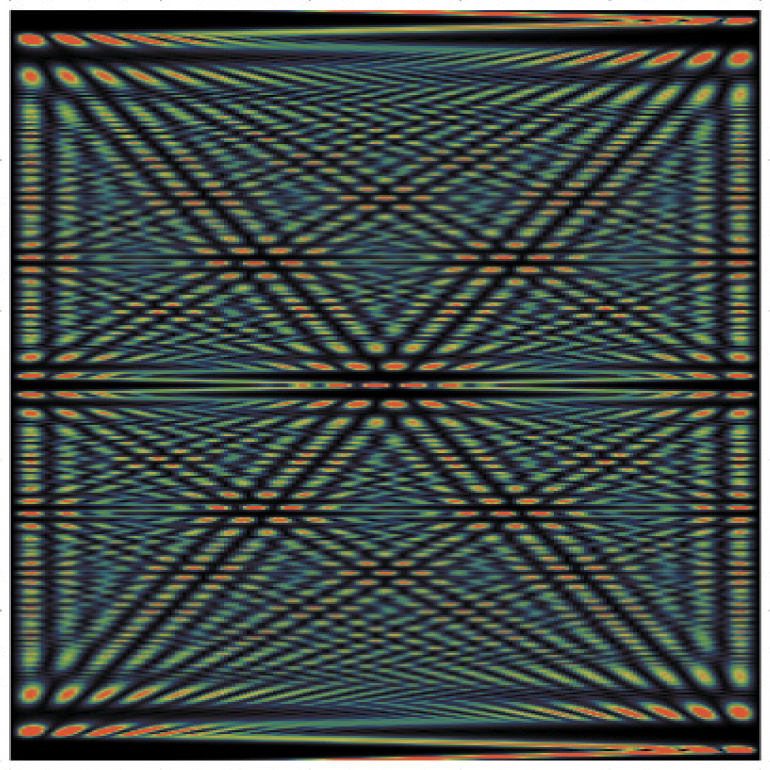}$$
{\small{Figure 5. Left: the optical effect (space and distance axis). Right: A quantum carpet i.e. plots of probability density for the propagation of a Gaussian wavepacket in a 1D box of length 1 on a revival time interval rescaled to 1 (space and time axis), from \cite{BerryMarzoliSchleich2001}.}}
\end{center}

This effect can be explained by the following fact. Adopting the quantum framework, the Dirac comb evolution through the linear Schr\"odinger equation, i.e. the fundamental solution of the periodic Schr\"odinger equation, is, by using Poisson summation formula:
$$e^{it\Delta}(\sum _{k\in\mathbb Z}\delta_k)(x)=e^{it\Delta}(\sum_{k\in\mathbb Z}e^{i2\pi kx})=\sum_{k\in\mathbb Z}e^{-it(2\pi k)^2+i2\pi kx},$$
so for $t_{p,q}=\frac 1{2\pi}\frac pq$ we have 
$$e^{it_{p,q}\Delta}(\sum _{k\in\mathbb Z}\delta_k)(x)=\frac 1{q}\sum_{l\in\mathbb Z}\sum_{m=0}^{q-1}G(-p,m,q)\delta_{l+\frac mq},$$
where $G(-p,m,q)$ stands for the Gauss sum
$$G(-p,m,q)=\sum_{l=0}^{q-1}e^{2\pi i\frac{-pl^2+ml}{q}}.$$

The linear and nonlinear Schr\"odinger evolution on the torus of functions with bounded variation was proved to present Talbot effect features by Berry \cite{Berry1996}, Berry and Klein \cite{BerryKlein1996}, Oskolkov \cite{Oskolkov1992}, Kapitanski  and Rodnianski \cite{KapitanskiRodnianski1999}, Rodnianski \cite{Rodnianski2000}, Erdogan and Tzirakis \cite{ErdoganTzirakis2013}). We note that also other dispersive equations enjoy Talbot effect, as shown for instance by Boulton, Farmakis, Pelloni and Smith \cite{BoultonFarmakisPelloniSmith2024}. 

In the Dirac deltas more singular setting on the torus, De la Hoz and Vega, by supposing uniqueness of solutions, founded in \cite{DelaHozVega2014} a solution of the 1D cubic Schr\"odinger equation \eqref{CubicNLS} similar to the fundamental solution above, thus exhibiting a Talbot effect. Then they constructed curves, which have as filament function this solution at rational times $t_{p,q}=\frac 1{2\pi}\frac pq$, that are skew polygons with corners located at $\frac 1q$ distances, thus again a Talbot effect.  

In \cite{BanicaVega2020} we placed ourselves in the Dirac deltas singular setting on $\mathbb R$. We proved a Talbot effect for the 1D cubic Schr\"odinger equation \eqref{CubicNLS} at this rough regularity. To do so we first noted a Talbot effect at the level of linear solutions, as follows. 

\begin{theorem}\label{thTalbotlin}(A Talbot effect for linear evolutions of Dirac type, \cite{BanicaVega2020})
Let $p\in\mathbb N, \eta\in(0,1)$ and $u_0$ be such that $\hat{u_0}$ is $2\pi-$periodic with $\hat u_0$ supported modulo $2\pi$ in $B(0,\eta\frac {\pi}{ p})$. For $x\in\mathbb R$ we define:
{\small{$$\xi_{x} :=\frac {\pi q}p \,dist\left (x,\frac 1q\mathbb Z\right) \in[0,\frac \pi p).$$}}
Then, for $t_{p,q}=\frac 1{2\pi}\frac pq$ with $q$ odd there exists $\theta_{x,p,q}\in\mathbb R$ such that:
{\small{$$e^{it_{p,q}\Delta}u_0(x)=\frac {1}{\sqrt{q}} \,\hat{u_0}(\xi_x)\, e^{-it_{p,q}\,\xi_x^2+ix\,\xi_x+i\theta_{x,p,q}}.$$}}
In particular, $e^{it_{p,q}\Delta}u_0(x)$ has $\frac 1q$-periodic modulus and vanishes if $d(x,\frac 1q\mathbb Z)>\frac \eta{q}$. 

Moreover, linear evolutions can concentrate near $\frac1q\mathbb Z$. More precisely, there is a family of initial data $u_0^\lambda=\sum _{k\in\mathbb Z}\alpha_k^\lambda\delta_k$ such that there exists $C>0$ and:
$$
\left|\frac{e^{it_{p,q}\Delta}u_0^\lambda(0)}{e^{it_{p,q}\Delta}\alpha_0^\lambda\delta_0(0)}\right|=C\frac{\sqrt{p}}{q}\lambda\overset{\lambda\rightarrow\infty}{\longrightarrow}\infty.$$
\end{theorem}
At the end of this section we shall give the sketch of the proof.  
The concentration phenomena is obtained by taking a sequence of initial data $\{u_0^\lambda\}$ that focus in Fourier variable near $2\pi$-integers: 
$$\widehat{u_0^\lambda}_{|(-\pi,\pi)}(\xi)=\lambda\psi(\lambda \xi)=\sum_k\alpha_k^\lambda e^{ik\xi},$$ 
with $\psi$ a bump function. This can be seen as a sequence approaching a Dirac comb. We note that the Dirac comb satisfies the hypothesis of the periodicity and localization in Fourier, since it equals to its Fourier transform. Therefore we recover the known Talbot effect explained above. However, this kind of data does not enter the nonlinear framework. Nevertheless, we obtain a Talbot effect also at the nonlinear level, at the same rough regularity, by using the solutions in Theorem \ref{thevpoly} that have as a leading term a linear evolution of a sum of Dirac deltas. 

\begin{theorem}\label{thTalbot}(A Talbot effect for nonlinear evolutions of Dirac type, \cite{BanicaVega2020})
Let $p\in\mathbb N, \eta\in(0,1)$, $s>\frac 12$ and $\epsilon\in (0,1)$. Let $u_0$ such that $\hat u_0$ is a $2\pi-$periodic, supported modulo $2\pi$ in $B(0,\eta\frac {\pi}{ p})$ , and having a Fourier coefficient sequence $\{\alpha_k\}$ satisfying $\|\{\alpha_k\}\|_{ l^{2,s}}:=(\sum_k|\alpha_k|^2(1+k^2)^s)^\frac 12=\epsilon$. 
 Let $u(t,x)$ be the solution of the 1D cubic Schr\"odinger equation \eqref{CubicNLS} on $(0,T)$ obtained in Theorem \ref{thNLScrit} from the sequence $\{\alpha_k\}$. Then for $t_{p,q}=\frac 1{2\pi}\frac pq$ with $1<q$ odd such that $\epsilon^2\sqrt{q}\log q<\frac 12$ the function $u(t,x)$ almost vanishes, in the sense:
$$|u(t_{p,q},x)|\leq \epsilon,\quad \mbox{if } d(x,\frac 1q\mathbb Z)>\frac \eta{q}.$$

Moreover, nonlinear solutions can concentrate near $\frac1q\mathbb Z$ in the sense that there is a family of sequences $\{\alpha_k^\lambda\}$ with $\|\{\alpha_k^\lambda\}\|_{l^{2,s}}\overset{\lambda\rightarrow \infty}{\longrightarrow}0$, such that the solutions $u^\lambda$ of \eqref{CubicNLS} obtained in Theorem \ref{thNLScrit} from the sequence $\{\alpha_k\}$ satisfy:
$$
\left|\frac{u^\lambda(t_{p,q},0)}{e^{it_{p,q}\Delta}\alpha_0^\lambda\delta_0(0)}\right|\overset{\lambda\rightarrow\infty}{\longrightarrow}\infty.
$$
\end{theorem}
Let us notice that despite the smallness condition on $\{\alpha_k\}$, the corresponding solution $u(t)$ from Theorem \ref{thNLScrit} is large for small times in $L^\infty$ and  $L^1_{loc}$, due to its $\frac 1{\sqrt t}$ factor. Going either backward or forward in time this, combined with the Talbot effect in Theorem \ref{thTalbot}, gives a phenomenon of constructive/destructive interference. More precisely, in Theorem 5.1 in \cite{BanicaVega2024} we proved the existence of rational times $\frac 1{2\pi}\frac pq$ and $\frac 1{2\pi}\frac {\tilde p}{\tilde q}$ such that on the interval $[-\frac1{2\tilde q},\frac 1{2\tilde q}]$ we observe at $\frac 1{2\pi}\frac pq$ almost-periodic small waves and at $\frac 1{2\pi}\frac {\tilde p}{\tilde q}$ a localized large-amplitude structure emerges. On one hand this can be seen as a creation of rogue waves in the sense of the dispersive blow-up, i.e. $L^\infty$-pointwise blow-up. This phenomena of dispersive blow-up was proved by Bona and Saut in \cite{BonaSaut2010} for the linear and the cubic Schr\"odinger equation on $\mathbb R$ by using smooth data of type $e^{ix^2}(1+x^2)^{-m}$ with $m\in (\frac 14,\frac 12]$. On the other hand, by using the pseudo-conformal transformation, our rogue waves result transfers to the periodic cubic Schr\"odinger equation \eqref{NLSt}, equation that has a coefficient $\frac 1t$ in front of the nonlinearity. In the context of the classical periodic cubic Schr\"odinger equation with random data recent progress in the rogue waves phenomena was done by Garrido, Grande, Kurianski and Staffilani in \cite{GarridoGrandeKurianskiStaffilani2023}. 

Finally, as a consequence of Theorem \ref{thTalbot}, and by using Theorem \ref{thevpoly}, we constructed in \cite{BanicaVega2020} families of binormal flow evolutions of polygonal lines that have at rational-type times $2\pi\frac pq$ curvature concentrating more and more at $\mathbb Z/q$ and getting smaller and smaller in between such locations.

\subsubsection{Sketch of the proof of Theorem \ref{thTalbotlin}}
Since $\hat{u_0}(\xi)$ is periodic we write:
$$e^{it\Delta}u_0(x)=\frac 1{2\pi}\int_{-\infty}^\infty e^{ix\xi}e^{-it\xi^2}\hat{u_0}(\xi)\,d\xi=\frac 1{2\pi}\sum_{k\in\mathbb Z}\int_{2\pi k}^{2\pi(k+1)} e^{ix\xi-it\xi^2}\hat{u_0}(\xi)\,d\xi$$
$$=\frac 1{2\pi}\int_0^{2\pi}\hat{u_0}(\xi)\sum_{k\in\mathbb Z}e^{ix(2\pi k+\xi)-it(2\pi k+\xi)^2}\,d\xi$$
$$=\frac 1{2\pi}\int_0^{2\pi}\hat{u_0}(\xi)e^{-it\xi^2+ix\xi}\sum_{k\in\mathbb Z}e^{-it\,(2\pi k)^2+i2\pi k (x- 2t\xi )}\,d\xi.$$
For $t_{p,q}=\frac 1{2\pi}\frac pq$ we get using the linear periodic Talbot effect:

$$e^{it_{p,q}\Delta}u_0(x)=\frac 1{q}\int_0^{2\pi}\hat{u_0}(\xi)e^{-it_{p,q}\xi^2+ix\xi}\sum_{l\in\mathbb Z}\sum_{m=0}^{q-1}G(-p,m,q)\delta(x-2t_{p,q}\xi-l-\frac mq)\,d\xi.$$
Since $\hat u_0$ is located modulo $2\pi$ in $B(0,\eta\frac {\pi}{ p})$ then 
$$e^{it_{p,q}\Delta}u_0(x)=\frac {1}{\sqrt{q}} \,\hat{u_0}(\xi_x)\, e^{-it_{p,q}\,\xi_x^2+ix\,\xi_x+i\theta_{x,p,q}},$$
for some $\theta_{x,p,q}\in\mathbb R$ and $\xi_{x} :=\frac {\pi q}p\,d(x,\frac1q\mathbb Z) \in[0,\frac \pi p).$

For proving the concentration effect of Theorem \ref{thTalbotlin} we shall construct a family of sequences $\{\alpha_k^\lambda\}$ such that $\sum _{k\in\mathbb Z} \alpha_k^\lambda \delta_k$ concentrates in the Fourier variable near the integers. To this purpose we consider $\psi$ a real bounded function with support in $[-\frac 12,\frac 12]$ and $\psi(0)=1$. We define
$$f^\lambda(\xi)= \lambda\psi(\lambda\xi) ,\forall \xi\in[-\pi,\pi].$$
Thus we can decompose
$$f^\lambda(\xi)=\sum _{k\in\mathbb Z} \alpha_k^\lambda e^{ik\xi},$$
and consider 
$$u_0^\lambda=\sum _{k\in\mathbb Z} \alpha_k^\lambda \delta_k.$$
In particular, on $[-\pi,\pi]$, we have $\widehat{u_0^\lambda}=f^\lambda$. Given $t_{p,q}=\frac 1{2\pi}\frac pq$, for $\lambda>p$, the restriction of $\widehat{u_0^\lambda}$ to $[-\pi,\pi]$ has support included in $B(0,\eta\frac \pi p)$ for a $\eta\in]0,1[$. We then get by the first part of the statement
$$e^{it_{p,q}\Delta}u_0^\lambda(0)=\frac {1}{\sqrt{q}} \,\widehat{u_0^\lambda}(0)\, e^{-it_{p,q}\,\xi_x^2+ix\,\xi_x+i\theta_{x,p,q}},$$
so 
$$|e^{it_{p,q}\Delta}u_0^\lambda(0)|=\frac {1}{\sqrt{q}} \,|f^\lambda(0)|=\frac {1}{\sqrt{q}} \lambda\psi(0)=\frac {1}{\sqrt{q}} \lambda.$$
On the other hand, at $t_{p,q}=\frac 1{2\pi}\frac pq$ we have
$$|e^{it_{p,q}\Delta}\alpha_0^\lambda\delta_0(0)|=\sqrt{\frac {4q}p} \,|\alpha_0^\lambda|=\sqrt{\frac {4q}p} \,\frac1{2\pi}\left|\int_{-\pi}^\pi f^\lambda(\xi)d\xi\right|=C(\psi)\sqrt{\frac qp} \lambda^{-1}.$$
Therefore,
$$\left|\frac{e^{it_{p,q}\Delta}u_0^\lambda(0)}{e^{it_{p,q}\Delta}\alpha_0^\lambda\delta_0(0)}\right|=\frac{
\sqrt{p}}{C(\psi)q}\lambda\overset{\lambda\rightarrow\infty}{\longrightarrow}\infty,$$  
and the proof of Theorem \ref{thTalbotlin} is complete.

\section{Intermittency and multifractality}
As recalled, in numerical simulations of the binormal flow evolution of a $M-$regular polygon, the trajectories in time of corners $\chi_M(t,0)$ were showed to behave as Riemann's complex function 
\begin{equation}\label{Riemannf}
R(t)=\sum_{k\in\mathbb Z}\frac{e^{itk^2}-1}{k^2},
\end{equation}
when $M$ tends to infinity (De la Hoz and Vega \cite{DelaHozVega2014}, De la Hoz, Kumar and Vega  \cite{delaHozKumarVega2020}, see Figure 4). In \cite{BanicaVega2022} we showed this kind of behavior by considering sequences of polygonal lines. More precisely, let $n\in\mathbb N^*$, $\nu\in]0,1]$, $\theta>0$. We consider polygonal lines $\chi_n(0)$ with finite but possibly many corners located at $j\in\mathbb Z$ with $|j|\leq n^\nu$    
and curvature angles $\theta_n$ such that 
$$
\theta_n=\pi-\frac{\theta}n+o\Big(\frac 1n\Big),
$$
$\chi_n(0,0)=0_{\mathbb R^3}$, and $\chi_n(0)$ lying in the $xy$-plane and symmetric with respect to the $yz-$plane, see Figure 6 for examples.
\begin{center}
\includegraphics[width=5in]{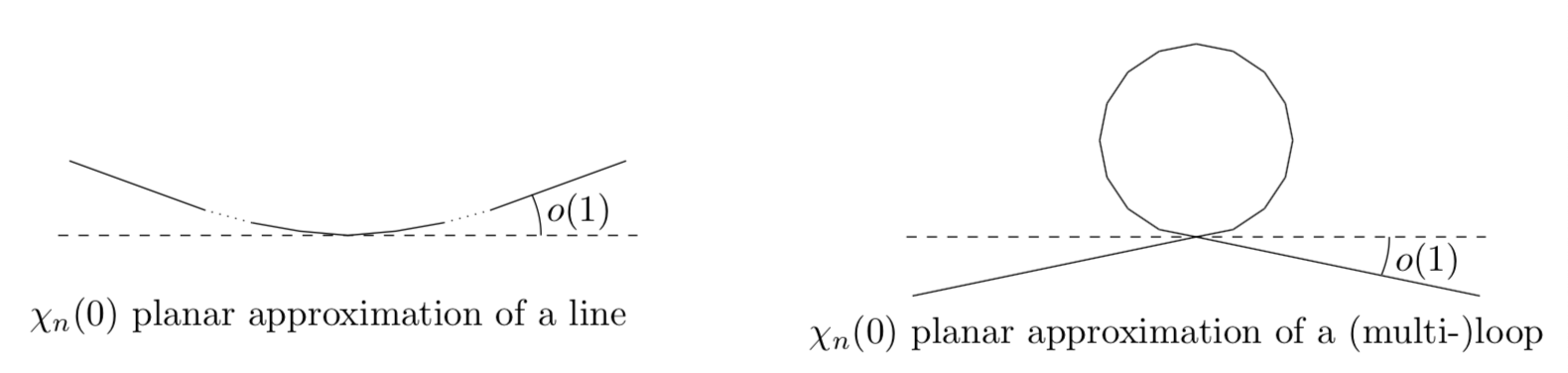}\\
{\small{Figure 6. Examples of curves $\chi_n$. }}
\end{center}
These polygonal lines with finite number of corners enter the framework of Theorem \ref{thevpoly}, but apriori their time of existence depends on $n$. We first refined in \cite{BanicaVega2022} the general analysis of Theorem \ref{thevpoly} to the specific case of $\chi_n(0)$. We show that there exists $T>0$ and $n_0\in\mathbb N^*$, both depending only on $\theta$, such that for all $n\geq n_0$ there exist smooth solutions $\chi_n(t)$ of the binormal flow on $(-T,T)\setminus\{0\}$, that are weak solutions on $(-T,T)$, and at time $t=0$ converge pointwise  to $\chi_n(0)$. Then we proved the following result. 
\begin{theorem}\label{thmultifrac}(A multifractal trajectory, \cite{BanicaVega2022})
We have the following description of the trajectory of the corner $\chi_n(t,0)$, uniformly on $(0,T)$: 
\begin{equation}\label{cv}
n(\chi_{n}(t,0)-\chi_{n}(0,0))-\theta(0,\Re({R}(t)),\Im({R}(t)))\overset{n\rightarrow\infty}\longrightarrow 0.
\end{equation}
\end{theorem}
This  theorem gives a non-obvious non-linear geometric interpretation of Riemann's function. It is valid also for $\chi_n(0)$ with same torsion angle $\omega_0$ at all corners, and then the limit is given by the following Riemann-type function:
$$\frak R_{\omega_0}(t)=\sum_{k\in\mathbb Z}\frac{e^{it(k-\omega_0)^2}-1}{(k-\omega_0)^2}.$$
We shall give the sketch of the proof of Theorem \ref{thmultifrac} in \S\ref{section-multifrac}. For rational-type torsion $\omega_0$ we also computed the spectrum of singularities of $\frak R_{\omega_0}$, showed that it is the same as the one of Riemann's function, and proved that $\frak R_{\omega_0}$ satisfies the Frisch-Parisi formalism\footnote{The Frisch-Parisi multifractal formalism was originally proposed for the velocity in an Eulerian setting, 
but it can be equally proposed in the Lagrangian setting, 
to which Riemann's function is in our context more related to since it represents a time trajectory.
 See the work of Chevillard et al. \cite{ChevillardCastaignArneodoLevequePintonRoux2012} for a discussion on the differences between these two frameworks.} and intermittency, notions that we recall in the following (see also the book of Frisch \cite{Frisch1995}). 

We recall that Riemann's function $\sum_{k\in\mathbb Z^*}\frac{\sin(tk^2)}{k^2}$, and implicitly its complex versions $\sum_{k\in\mathbb Z^*}\frac{e^{itk^2}}{k^2}$ 
and $R$ defined in \eqref{Riemannf}, was studied by Jaffard in 1996 (\cite{Jaffard1996}) from the point of view of multifractal analysis, first by computing its spectrum of singularities. The spectrum of singularities of a function $f:[0,1]\rightarrow\mathbb R$ is given by the Hausdorff dimension of non-empty iso-H\"older sets,  
$$d_{f}(\alpha)=\dim_\mathcal H\{t\in[0,1],\sup\{\beta, f\in\mathcal C^\beta(t)\}=\alpha\},$$
with the convention $d_f(\alpha)=-\infty$ if the local $\alpha$-H\"older regularity is not reached. Jaffard (see also Brouke and Vindas (\cite{BrouckeVindas2023}) for a recent different proof of computing pointwise the H\"older regularity) showed that
$$d_R(\alpha)=\left\{\begin{array}{c}4\alpha-2,\alpha\in[\frac 12,\frac 34],\\0,\alpha=\frac 32,\\-\infty,\mbox{ otherwise}. \end{array}\right.$$
 Moreover, he proved that it satisfes the multifractal formalism of Frisch-Parisi, a model for predicting the structure function exponents in turbulent flows, that was motivated by the experiments in Modane of Anselmet and all in 1984 (\cite{AnselmetGagneHopfingerAntonia1984}) showing deviations from Kolmogorov 41 theory:
$$
d_{{R}}(\alpha)=\inf_p (\alpha p-\eta_{{R}}(p)+1),\quad\forall\alpha\in\Big[\frac 12,\frac 34\Big],$$
where $\eta_R$ is determined in terms of Besov spaces\footnote{We recall first the definition of high-pass filters. Let $\Phi \in C^\infty (\mathbb R)$ be a cutoff function vanishing in a neighborhood of the origin 
and such that $\Phi(x) = 1$ for $|x| \geq 2$.  
For a periodic function $f(t) = \sum_{k \in \mathbb Z} a_k e^{2\pi i k t}$, and for $N \in \mathbb N$, we define the  high-pass filter of Fourier modes larger than $N$ as:
$$
P_{\geq N} f(t) = \sum_{k \in \mathbb N} \Phi\Big(  \frac{k}{N} \Big) \,  a_k \, e^{2\pi i k t }.
$$ Similarly we consider the band-pass filter $P_{\simeq N} f$ to be defined with the cutoff $\Phi$ satisfying the additional assumption of compact support. 
Finally we recall that $f\in B^{\frac sp}_{p,\infty}$ if and only if $\{2^{k\frac sp}\|P_{\simeq 2^k}f\|_p\}_{k\in\mathbb Z}\in\ell^\infty$.}:
$$\eta_{{R}}(p):=\sup\{s,\,{R}\in B^{\frac sp}_{p,\infty}\}.$$
We note that recently Barral and Seuret  (\cite{BarralSeuret2023}-\cite{BarralSeuret2023bis}) proved the validity of the multifractal formalism generically in Besov spaces. Again having in mind turbulent dynamics, Boritchev, Eceizabarrena and Da Rocha (\cite{BoritchevDaRochaEceizabarrena2021}) proved that $R$ is intermittent in small scales by showing that the flatness  satisfies:
\begin{equation}\label{flatness}
F_R(N):=\frac{\|P_{\geq N}R\|^4_4}{\|P_{\geq N}R\|^4_2}\overset{N\rightarrow\infty}{\rightarrow}\infty,
\end{equation} 
where $P_{\geq N}$ is the high-pass filter of Fourier modes larger than $N$ defined in the previous footnote.\medskip

In collaboration with Eceizabarrena and Nahmod we extended in \cite{BanicaEceizabarrenaNahmodVega2024} Theorem \ref{thmultifrac} to all locations $x_0\in\mathbb R$, showing that the convergence \eqref{cv} holds with
$$n(\chi_{n}(t,x_0)-\chi_{n}(0,x_0))-\theta(0,\Re({R_{x_0}}(t)),\Im({R_{x_0}}(t)))\overset{n\rightarrow\infty}\longrightarrow 0,
$$
where
$$ R_{x_0}(t)=\sum_{k\in\mathbb Z}\frac{e^{itk^2}-1}{k^2}e^{ikx_0}.$$ 
In Figure 7 several such trajectories are represented, among which, for $x_0=0$, the one of the classical Riemann's complex function \eqref{Riemannf}.

\begin{center}
$$\includegraphics[height=1.5in]{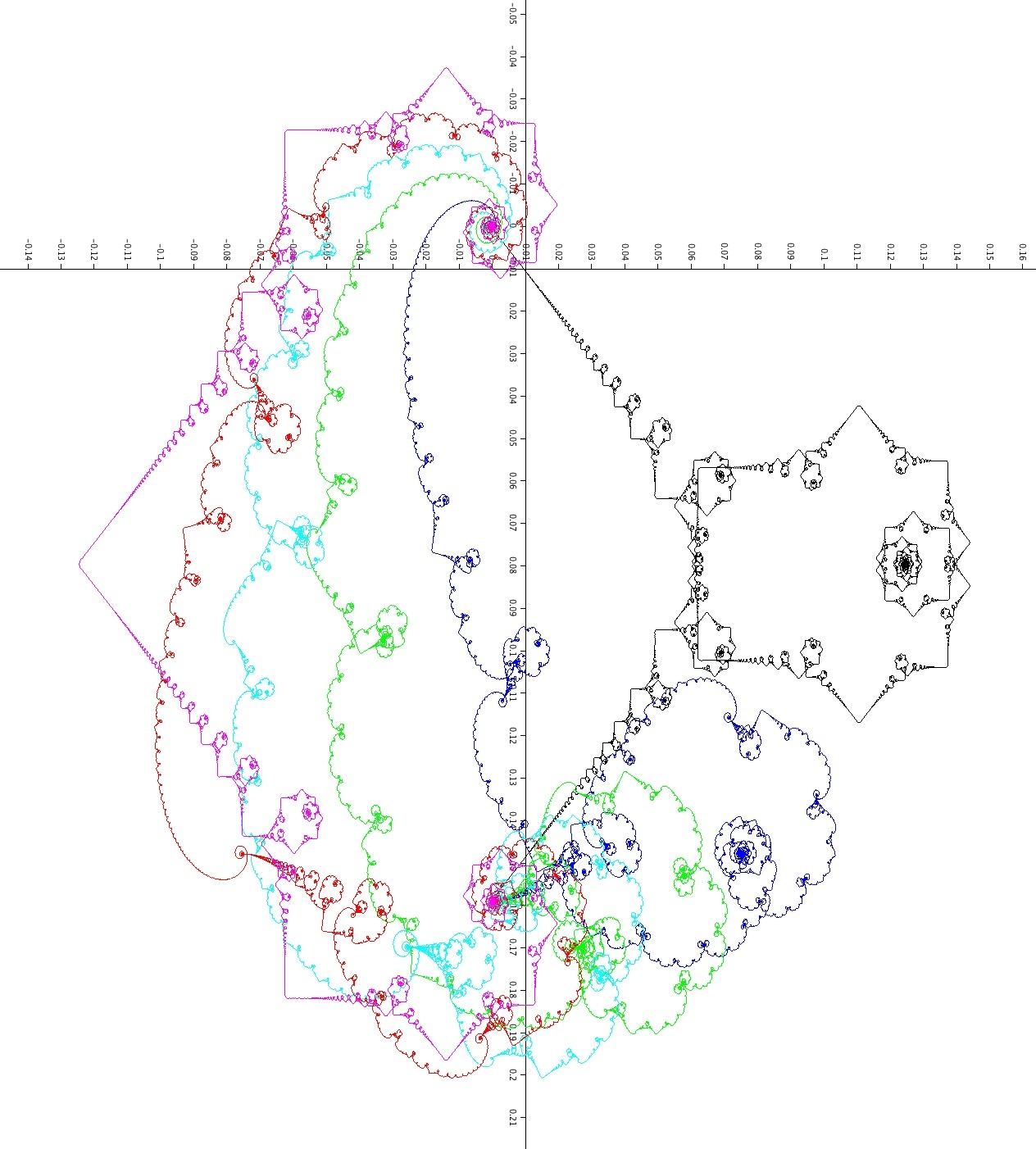}$$
{\small{Figure 7. Trajectories of the limit curve evolution $R_{x_0}$ for $t\in[0,2\pi]$, at $x_0=0$ (black), $0.1$ (blue), $0.2$ (green), $0.3$ (cyan), $0.4$ (red), $0.5$ (magenta).}}
\end{center}

Many multifractal analysis studies have been done for various types of modifications of Riemann's function. For instance Chamizo and Ubis (\cite{ChamizoUbis2014}) and Seuret and Ubis (\cite{SeuretUbis2017}) studied $\Sigma_{k}\frac{e^{iP(k)t}}{k^\alpha}$ for $P$ polynomial. Also, Kapitanski and Rodnianski studied in \cite{KapitanskiRodnianski1999} the behavior in space of $u_{t_0}(x)=\Sigma_{k}e^{ik^2t_0+ikx}$ for $t_0$ fixed.

In \cite{BanicaEceizabarrenaNahmodVega2024} we have studied the function $R_{x_0}$, that is a natural extension of Riemann's function, linked with the fundamental solution of the periodic Schr\"odinger equation, and that came out in the study of trajectories of binormal flow evolutions, having in mind that the binormal flow is a model for non-smooth fluids presenting vortex filaments. 

\begin{theorem}\label{thBENV}{(General multifractal study, \cite{BanicaEceizabarrenaNahmodVega2024})}
For $x_0\in\mathbb R$ the function $R_{x_0}$
is multifractal, with infinitely many local H\"older exponents. \\\vspace{2mm}
If $x_0\in\mathbb Q$ then $R_{x_0}$ has the same spectrum of singularities as Riemann's function:
$$d_{R_{x_0}}(\alpha)=\left\{\begin{array}{l}4\alpha-2, \alpha\in [\frac 12,\frac 34],\\ 0,\quad \alpha=\frac 32,\\-\infty,\mbox{ otherwise},\end{array}\right.$$
it satisfies the multifractal formalism of Frisch-Parisi, and it is intermittent in small scales.
\end{theorem}

The spectrum proof starts as in \cite{Jaffard1996} and then follows the approach in \cite{ChamizoUbis2014} but ends up with new Diophantine sets that approximate the iso-H\"older sets. We measure these new sets using Duffin-Schaeffer theorem from 1941 finalized by Koukoulopoulos and Maynard in 2020 (\cite{DuffinSchaeffer1941}-\cite{KoukoulopoulosMaynard2020}) and the Mass Transference Principle proved by Beresnevich and Velani in 2006 (\cite{BeresnevichVelani2006}). We shall give a sketch of the computation of the spectrum in \S \ref{section-spectrum}. 

To show that $R_{x_0}$ satisfies Frish-Parisi multifractal formalism we compute 
$\eta_{{R_{x_0}}}(p)$
by obtaining the $L^p$ norms of dyadic blocs of $R_{x_0}'$, based on $L^\infty$ estimates of partial sums of the fundamental solution of periodic 1D Schr\"odinger equation, i.e. Vinogradov trigonometrical series, near rational time and space locations, together with counting techniques.

\subsubsection{Sketch of the proof of Theorem \ref{thmultifrac}}\label{section-multifrac}
The solutions of the binormal flow $\chi_n$ were constructed by using Hasimoto's method from the following solutions \eqref{superpsol} in Theorem \ref{thNLScrit}:
$$u_{\{\alpha_{n,k}\}}(t,x)=\sum_{k\in\mathbb Z}e^{- i(|\alpha_{n,k}|^2-2\sum_{j\in\mathbb Z}|\alpha_{n,j}|^2)\log t}(\alpha_{n,k}+R_{n,k}(t))e^{it\Delta}\delta_k(x),$$
 with 
$$\sup_{0<t<T}t^{-\gamma}\|\{R_{n,k}(t)\}\|_{l^{2,s}}+t\,\|\{\partial_t R_{n,k}(t)\}\|_{l^{2,s}}<C,$$
and $s>\frac 12, 0<\gamma<1$. In view of the particular shape of $\chi_n(0)$ we have $\alpha_{n,j}=c_n\approx \frac \theta n$. Using the equations \eqref{ev} we can compute the evolution of one arclength location:
$$\chi_n(t,0)-\chi_n(0,0)=\int_0^{t}\Im(\overline{u_{\{\alpha_{n,k}\}}}N_n(\tau,0))\,d\tau$$
$$=\Im\int_{0}^{t}\sum_k(\overline{\alpha_{n,k}+R_{n,k}(\tau)})\frac{e^{-i\frac{k^2}{4\tau}}}{\sqrt{\tau}}e^{- i(|\alpha_{n,k}|^2-2\sum_{j\in\mathbb Z}|\alpha_{n,j}|^2)\log t}\,N_n(\tau,0)\,d\tau.$$
The $R_{n,k}(\tau)$ constribution is $o(\frac 1n)$ due to the decay in time of its $l^1$ norm. In \cite{BanicaVega2020} we proved that a modulation of the normal vector $N_n(\tau,0)$ has a limit 
$\tilde N_n(0,0)\in\mathbb S^2+i\mathbb S^2$ at $\tau=0$, 
that allows us to estimate:
$$\chi_n(t,0)=\Im (e^{-i\sum_{1\leq |j|\leq n^\nu}|\alpha_{n,j}|^2\log|k|} \tilde N_n(0,0)\int_{0}^{t}\sum_{|k|\leq n^\nu}\overline{\alpha_{n,k}}\frac{e^{-i\frac{k^2}{4\tau}}}{\sqrt{\tau}}\,d\tau)+o(\frac 1n).$$
Recalling that $\alpha_{n,k}=c_n\approx \frac \theta n$ and $\nu\in]0,1]$ we obtain:
$$\chi_n(t,0)=\frac \theta n \Im (\tilde N_n(0,0)\int_{0}^{t}\sum_{|k|\leq n^\nu}\frac{e^{-i\frac{k^2}{4\tau}}}{\sqrt{\tau}}\,d\tau)+o(\frac 1n).$$
On one hand, we prove convergence of the modulated normal vectors at $(t,x)=(0,0)$:
$$\lim_{n\rightarrow\infty}\tilde N_n(0,0)=(0,\frac{1-i}{\sqrt{2}},\frac {-1-i}{\sqrt{2}}).$$
On the other hand, the summation in $k$ can be taken over the whole set $\mathbb Z$ and thanks to Poisson's summation formula $\sum_{k\in\mathbb Z}f(k)=\sum_{j\in\mathbb Z}\hat f(2\pi j)$ we have: 
$$\sum_{k\in\mathbb Z}e^{i 4\pi^2 tk^2}=\sum_{k\in\mathbb Z}\int e^{-i2\pi xk+i4\pi^2 tx^2}dx=\frac 1{\sqrt{4\pi^2 t}}\sum_{j\in\mathbb Z}\int e^{-iy\frac j{\sqrt{t}}+iy^2}dy$$
$$=\frac 1{2\pi \sqrt{ t}}\sum_{j\in\mathbb Z}\widehat{e^{i\cdot^2}}(\frac j{\sqrt{t}})=\frac {e^{i\frac \pi 4}}{2\sqrt{\pi}\sqrt{t}}\sum_{j\in\mathbb Z}e^{- i\frac{j^2}{4t}},$$
Thus, neglecting again the $2\pi$-factors for simplifying the presentation, we obtain uniformly on $(0,T)$:
$$n\,\chi_{n}(t,0)-\theta(0,\Re(R(t)),\Im(R(t)))\overset{n\rightarrow\infty}\longrightarrow 0.$$

\subsubsection{Sketch of the proof of Theorem \ref{thBENV} for $\alpha\in(\frac 12,\frac 34]$.}\label{section-spectrum}
We start by analyzing the variation of $R_{x_0}$ at rationals. 
Splitting the sum in $n$ modulo $q$ and using Poisson summation formula we get:
$$R_{x_0}\Big(\frac pq+h\Big)-R_{x_0}\Big(\frac pq\Big)=-ih+\sum_{n\in\mathbb Z} e^{in^2\frac pq }\frac{e^{in^2h}-1}{n^2}e^{inx_0}$$
$$=-ih+\frac{\sqrt{h}}{q}\sum_{m\in\mathbb Z}G(p,m,q) F\Big(\frac{x_0-\frac mq}{\sqrt{h}}\Big),$$
with $G(p,s,q)=\Sigma_{r=0}^{q-1}e^{i\frac{p}{q}r^2+i\frac{s}{q}r}$ i.e. Gauss sums, of size $\sqrt{q}$ except if $q$ even and $q/2,s$ are of diferent parity, 
and $F(x)=\mathcal F(\frac{e^{i\xi^2}-1}{\xi^2})=O(\frac 1{x^2}),$\vspace{2mm} $F(0)\neq 0$. Thus the leading term is given by $m_{x_0,q}$ s.t. $x_0-\frac {m_{x_0,q}}q=d(x_0,\frac{\mathbb Z}{q})$:
\begin{equation}\label{varrat}
R_{x_0}\Big(\frac pq+h\Big)-R_{x_0}\Big(\frac pq\Big)=\frac{\sqrt{h}}{q}G(p,m_{x_0,q},q) F\Big(\frac{d(x_0,\frac{\mathbb Z}{q})}{\sqrt{h}}\Big) -ih+O(\min\{\sqrt{q}h,q^\frac 32 h^\frac 32\}).
\end{equation}
Therefore, we obtain $R_{x_0}\in\mathcal C^\frac 12(\frac pq)$ if $G(p,m_{x_0,q},q)\neq 0$ and $d(x_0,\frac{\mathbb Z}{q})=0$. These conditions are satisfied for rationals $x_0=\frac PQ$ together with $q\in 4Q\mathbb Z$:
\begin{equation}\label{badH}
R_{\frac PQ}\in\mathcal C^\frac 12(\frac pq).
\end{equation}

Now we shall look for upper and lower bounds for H\"older regularity at irrationals. To do so we have to recall the notion of 
exponent of irrationality of $t\notin\mathbb Q$:
$$\mu(t)=\sup\{\nu, t\in {\bf{A_\nu\}}},$$ 
where
$${\bf{A_{\nu}}}=\{t\notin\mathbb Q,  |t-\frac pq|\leq \frac{1}{q^\nu} \mbox{ for infinitely many coprime pairs } (p,q)\in \mathbb N\times \mathbb N\}.$$
Also, we recall that the approximation by continuous fractions of $t$ satisfies
$$|t-\frac{p_n}{q_n}|=\frac1{q_n^{\mu_n}}\leq\frac 1{q_{n+1}q_n}\mbox{ and }\mu(t)=\limsup_{n\rightarrow\infty}\mu_n.$$
 Thus, for all $h$ small there exists $n$ such that
 $$|t-\frac{p_n}{q_n}|\leq h\leq |t-\frac{p_{n-1}}{q_{n-1}}|.$$ 
 
 For estimating the variations at $t$ we rely on variations at the rationals $\frac{p_n}{q_n}$, and use \eqref{varrat} to get:
$$|R_{x_0}(t+h)-R_{x_0}(t)|\leq |R_{x_0}(\frac{p_n}{q_n}+(t-\frac{p_n}{q_n}+h))-R_{x_0}(\frac{p_n}{q_n})|+|R_{x_0}(\frac{p_n}{q_n})-R_{x_0}(\frac{p_n}{q_n}+(t-\frac{p_n}{q_n}))|$$
$$\lesssim \frac{\sqrt{h}}{\sqrt{q_n}}+h+\min\{\sqrt{q_n}h,q_n^\frac 32h^\frac 32\}\lesssim h^{\frac 12+\frac{1}{2\mu_n}}+h^{\frac 12+\frac{1}{2\mu_{n-1}}}\lesssim h^{\frac 12+\frac{1}{2\mu}-\delta},\quad\forall\delta>0.$$
Therefore we have a lower bound for the H\"older exponent of $R_{x_0}$ at $t$:
\begin{equation}\label{lowerH}
\alpha_{R_{x_0}}(t)\geq\frac 12+\frac1{2\mu(t)}.
\end{equation}

To get an upper bound for H\"older regularity at irrational $t$ we consider sets of irrationals well approximated by rationals $\frac{p_n}{q_n}, q_n\in 4Q\mathbb Z$, where the H\"older regularity is $\mathcal C^\frac 12$ by \eqref{badH}:
$${\bf{A_{\mu,Q}}}=\{t\notin\mathbb Q,  |t-\frac pq|\leq \frac{1}{q^\mu} \mbox{ for infinitely many coprime pairs } (p,q)\in \mathbb N\times 4Q\mathbb N\}.$$
We note that in particular if $t\in {\bf{A_{\mu,Q}}}$ then $\mu\leq \mu(t)$. For $t\in{\bf{A_{\mu,Q}}}$ we consider $(p_n,q_n)\in \mathbb N\times 4Q\mathbb N$ from the definition of  ${\bf{A_{\mu,Q}}}$. We define $h_n,\nu_n$ such that
$$h_n=t-\frac{p_n}{q_n}, \quad \frac{1}{q_n^{\nu_n}}=|h_n|.$$
In particular we have $\nu_n\geq \mu$. 
We use \eqref{varrat}-\eqref{badH} to get the lower estimate:
$$|R_{x_0}(t+h_n)-R_{x_0}(t)|=|R_{x_0}(\frac{p_n}{q_n})-R_{x_0}(\frac{p_n}{q_n}+h_n)|\gtrsim \frac{\sqrt{h_n}}{\sqrt{q_n}}=h_n^{\frac 12+\frac1{2\nu_n}}\geq h_n^{\frac 12+\frac1{2\mu}}.$$
Hence we have obtained an upper-bound for the H\"older regularity at $t\in {\bf{A_{\mu,Q}}}$ that together with \eqref{lowerH} gives the constraint:
\begin{equation}\label{lowerupperH}
\frac 12+\frac1{2\mu}\geq\alpha_{R_{x_0}}(t)\geq\frac 12+\frac1{2\mu(t)},\quad \forall t\in {\bf{A_{\mu,Q}}}.
\end{equation}

Now we shall start approximating the iso-H\"older sets. We remove from ${\bf{A_{\mu,Q}}}$ the points that might have $\mu(t)>\mu$ by introducing the sets
$${\bf{B_{\mu,Q}}}={\bf{A_{\mu,Q}}}\setminus\Big(\cup_{\epsilon>0}{\bf{A_{\mu+\epsilon}}}\Big).$$
Then from ${\bf{B_{\mu,Q}}}\subset {\bf{A_{\mu}}}\setminus\Big(\cup_{\epsilon>0}{\bf{A_{\mu+\epsilon}}}\Big)$ and the definition of $\mu(t)$ we get 
\begin{equation}\label{irr}
\mu(t)=\mu,\quad \forall t\in {\bf{B_{\mu,Q}}}.
\end{equation}
Therefore, we obtain the following approximation of the iso-H\"older sets:
$${\bf{B_{\mu,Q}}}\subset \{t,\alpha_{R_{x_0}}(t)=\frac 12+\frac1{2\mu}\}\subset {\bf{A_{\mu-\epsilon}}}, \forall\epsilon>0.$$
Indeed, the first inclusion follows then by \eqref{lowerupperH}-\eqref{irr}. From the general lower bound \eqref{lowerupperH} it follows that a point $t$ such that $\alpha_{R_{x_0}}(t)=\frac 12+\frac1{2\mu}$ satisfies $\mu\leq \mu(t)$. Thus $\mu-\epsilon<\mu(t)=\sup\{\nu, t\in {\bf{A_\nu\}}}$ and the second inclusion follows from the definition of ${\bf{A_\nu}}$.

To prove the spectrum of singularities in Theorem \ref{thBENV} we need to show that for $\mu\geq 2$:
$$\dim_{\mathcal H}\{t,\alpha_{R_{x_0}}(t)=\frac 12+\frac1{2\mu}\}=\frac 2\mu.$$
Since
$${\bf{B_{\mu,Q}}}={\bf{A_{\mu,Q}}}\setminus\Big(\cup_{n}{\bf{A_{\mu+\frac 1n}}}\Big)\subset \{t,\alpha_{R_{x_0}}(t)=\frac 12+\frac1{2\mu}\}\subset {\bf{A_{\mu-\epsilon}}}.$$
and since Jarn\'{i}k-Besicovitch theorem proved in 1931-1934 (\cite{Jarnik1931}-\cite{Besicovitch1934}) states that $\dim_{\mathcal H}{\bf{A_\mu}}=\frac 2{\mu}$, it is enough to show that:
$$\dim_{\mathcal H}{\bf{B_{\mu,Q}}}\geq \frac 2\mu.$$
Moreover, as
$$\mathcal H^\frac 2\mu ({\bf{B_{\mu,Q}}})=\mathcal H^\frac 2\mu ({\bf{A_{\mu,Q}}})-\lim_{n\rightarrow\infty}\mathcal H^\frac 2\mu({\bf{A_{\mu+\frac 1n}}})=\mathcal H^\frac 2\mu ({\bf{A_{\mu,Q}}}),$$
it is enough to prove $\mathcal H^\frac 2\mu ({\bf{A_{\mu,Q}}})>0$. Thus, to finish the proof it is enough to show
$$\mathcal H^\frac 2\mu ({\bf{A_{\mu,Q}}})=+\infty.$$

To do so we start with a Lebesgue measure information. 
We use Duffin-Schaeffer theorem (\cite{DuffinSchaeffer1941}-\cite{KoukoulopoulosMaynard2020}) which states that if
$${\bf{A_\psi}}=\{ t, |t-\frac pq|\leq \psi(q)\mbox{ for infinitely many coprime pairs } (p,q)\in \mathbb N\times \mathbb N\},$$
and $\sum_{q\geq 1}\psi(q)\varphi(q)=\infty,$
where $\varphi$ is Euler's totient function\
 ($\varphi(q):=\# \{1\leq m\leq q, (m,q)=1\}$), then $|{\bf{A_\psi}}|=1$. 
 We show that $\psi(q):=\frac{\mathbb I_{4Q\mathbb N}(q)}{q^2}$ satisfies the hypothesis, so we get
\begin{equation}\label{Leb}
|{\bf{A_{2,Q}}}|=|{\bf{A_\psi}}|=1.
\end{equation}

Then we use transform this Lebesgue information into a Hausdorff information by the Mass Transference Principle in \cite{BeresnevichVelani2006}, which states the following.  
Let $B_n(x_n,r_n)$ a sequence of balls in $[0,1]$ with $r_n\rightarrow 0$ and $\alpha<d$. If 
$$|\limsup_n B_n(x_n, r_n^{\alpha})|=1,$$
then 
$$\dim_{\mathcal H}\limsup_n B_n(x_n, r_n) \geq \alpha,\quad \mathcal H^\alpha(\limsup_n B_n(x_n, r_n))=+\infty.$$
Since we have obtained in \eqref{Leb}
$$1=|{\bf{A_{2,Q}}}|=\Big|\limsup_q\cup_{p\leq q,(p,q)=1}B\Big(\frac pq, \frac{\mathbb I_{4Q\mathbb N}(q)}{q^2}\Big)\Big|,$$
we apply the above result with $\alpha=\frac 2\mu$ and radii $\frac{\mathbb I_{4Q\mathbb N}(q)}{q^2}$ to get:
$$+\infty=\mathcal H^\frac 2\mu \Big(\limsup_q\cup_{p\leq q,(p,q)=1}B\Big(\frac pq, \frac{\mathbb I_{4Q\mathbb N}(q)}{q^\mu}\Big)\Big)=\mathcal H^\frac 2\mu({\bf{A_{\mu,Q}}}).$$

\section{Well-posedness, phase blow-up and unique continuation at critical regularity}\label{section-blup}
As we recalled in \S \ref{section-several} the only known wellposedness results for 1D cubic Schr\"odinger equation are in the subscritical spaces with respect to the scaling, in $H^s$ for $s>-\frac12$ on the Sobolev scale, and $\mathcal FL^p$ for $p<+\infty$ on the Fourier-Lebesgue scale, i.e. $L^p$-regularity of the Fourier transform (\cite{Harrop-GriffithKillipVisan2024},\cite{Grunrock2005}).

In collaboration with Luc\`a and Tzvetkov  \cite{BanicaLucaTzvetkovVega2024} we displayed a subset in the supercritical spaces $H^s$ for $s<-\frac12$ and in the critical space $\mathcal FL^\infty$ in which we obtained a well-posedness result in the following sense.

\begin{theorem}\label{thblup}(A well-posed critical subset for 1D cubic NLS and phase blow-up, \cite{BanicaLucaTzvetkovVega2024})
Let $r>0$ and $u_1(x)$ satisfying the following periodicity property at the level of the Fourier transform on $\mathbb R$:
\begin{equation}\label{perdata}
\widehat{e^{-i\Delta}u_1}(\xi)=e^{i\xi^2}\widehat u_1(\xi) \in H^r(0,2\pi).
\end{equation}
Then there exists a unique solution of \eqref{CubicNLS} on $(0,1]$ such that $u(1)=u_1$ and 
$$\widehat{e^{-it\Delta}u(t)}(\xi)=e^{it\xi^2}\widehat u(t,\xi) \in H^r(0,2\pi), \forall t\in(0,1].$$

Moreover, by denoting $\{A_k(t)\}_{k\in\mathbb Z}$ the Fourier coefficients of $e^{it\xi^2}\widehat u(t,\xi) $, we have the existence of a sequence $\{\alpha_k\}_{k\in\mathbb Z}\in l^{2,r}$ such that
$$|A_k(t)-e^{i(|\alpha_k|^2-2\|e^{i\xi^2}\widehat {u_1}(\xi)\|_{L^2(0,2\pi)}^2)\log t}\alpha_k|\leq C(\|e^{i\xi^2}\widehat {u_1}(\xi)\|_{H^r(0,2\pi)})\,t, \quad \forall k\in\mathbb Z, t\in(0,1).$$
In particular, we have a blow-up in the sense that $u(t)$ falls out from the periodicity functional framework at $t=0$, as the Fourier coefficients $A_k(t)$ don't have a limit, due to the logarithmic phase loss. 

The dependence with respect to the initial data is continuous in the following sense. 
Let $\{u_{1,n}\}_{n\in\mathbb N}$ be a sequence of initial data satisfying \eqref{perdata} such that 
$
e^{i \xi^2}\widehat{u_{1,n}}(\xi) 
$
converges to 
$
e^{i \xi^2}\widehat{u}(1,\xi)
$
in $H^r(0,2\pi)$.  Then for every $t\in (0,1]$ the sequence 
$
e^{i t\xi^2}\widehat{u_n}(t,\xi) 
$
converges to 
$
e^{i t\xi^2}\widehat u(t,\xi)
$
in $H^r(0,2\pi)$, where $u_n$(t) is the solution of \eqref{CubicNLS} with data $u_{1,n}$ at $t=1$. 
We have that 
$
u\in C\big((0,1]; {\mathcal S}'(\mathbb R)\big)
$
but it blows-up at zero in the sense that 
$
\lim_{t\rightarrow 0^+} u(t,\cdot)
$
does not exist in ${\mathcal S}'(\mathbb R)$. 
\end{theorem}

Let us first note that for free Schr\"odinger solutions, if $e^{i\xi^2}\widehat u(1, \xi)$ is $2\pi-$periodic then $e^{it\xi^2}\widehat u(t, \xi)$ is also $2\pi-$periodic as
$$e^{it\xi^2}\widehat u(t, \xi)=e^{i\xi^2}\widehat u(1,\xi).
$$
For the 1D cubic Schr\"odinger equation \eqref{CubicNLS} the evolution law of of $\omega(t,\xi):=e^{it\xi^2}\widehat u(t, \xi)$
is:
$$
\omega_t(\eta, t)=\frac{i}{8\pi^3}\int e^{it2(\xi_1-\xi_2)(\xi_1-\eta)} \omega(t,\xi_1)\bar\omega(t,\xi_2)\omega(t,\eta-\xi_1+\xi_2) \,d\xi_1d\xi_2,
$$
thus compatible with periodicity of $\omega$; the solutions we construct  in Theorem \ref{thblup} are in this framework.

Phase blow up phenomena were encountered for the Schr\"odinger equation since the works of Merle in the 90s (\cite{Merle1992}, see also \cite{MerleRaphaelSzeftel2013}). Here we shall prove that this loss of phase does not affect the associated solutions of the binormal flow via Hasimoto's transform, and that they can be uniquely continued after the singularity time $t=0$.

Let us now note that the periodicity condition in the theorem translates as follows:
$$\hat u(t,\xi)=\sum_{k\in\mathbb Z} A_k(t)e^{-it\xi^2}e^{ik\xi}=\sum_{k\in\mathbb Z} A_k(t)\widehat{e^{it\Delta}\delta_k},$$
so $u(t,x)$ has the ansatz \eqref{superp}:
$$u(t,x)=\sum_kA_k(t)e^{it\Delta}\delta_k=\sum_kA_k(t)\frac{e^{i\frac{(x-k)^2}{4t}}}{\sqrt{t}}.$$
In particular $u(t)$ belongs to the supercritical spaces $H^s(\mathbb R)$ for $s<-\frac 12$, but not to the subcritical ones, and belongs also to the critical space $\mathcal F L^\infty(\mathbb R)$ provided that $r>\frac 12$, but not to the subcritical ones. Also, the data in Theorem \ref{thblup} is simply
$$u(1,x)=e^{ix^2}f(x),$$
with $f $ a $4\pi-$periodic function in $H^r(0,4\pi)$ with $r>0$.

Examples of solutions as in Theorem \ref{thblup} were given in Theorem \ref{thNLScrit}. We recall that by using the pseudo-conformal transformation 
$$
v(\tau,y):=\frac{e^{i\frac{y^2}{4\tau}}}{\sqrt {\tau}}\,\overline u(\frac 1\tau,\frac y{\tau}),
$$
equation \eqref{CubicNLS} is transformed into 
\begin{equation}\label{NLSt}
iv_t+v_{xx}+ \frac 1t|v|^2v=0.
\end{equation} 
Moreover, the ansatz \eqref{superp} translates into simply being in the periodic setting of \eqref{NLSt} with data in $H^r(0,2\pi)$ at time $t=1$. Thus the local in time well-posedness in Theorem \ref{thblup} is obtained by Bourgain's approach classical for the periodic 1D cubic Schr\"odinger equation, as the factor $\frac 1t$ in \eqref{NLSt} is harmless for this argument. The delicate point is to understand the asymptotics of the solution $v(t)$ when $t$ goes to infinity. 
In particular the result of Theorem \ref{thNLScrit} is the existence of wave operators for the above periodic 1-D cubic NLS equation with time-variable coefficient  \eqref{NLSt}. In turn Theorem \ref{thblup} corresponds to the much more delicate result of asymptotic completeness for \eqref{NLSt}, which translates to a blow-up for \eqref{CubicNLS}. We shall briefly sketch the proof at the end of this section. \\

With the help of Hasimoto's construction we obtained as a consequence of Theorem \ref{thblup} and of the analysis in \cite{BanicaVega2020} the following result. 
\begin{theorem}\label{thcritbf}(Criterium for generating binormal flow singularities, \cite{BanicaLucaTzvetkovVega2024})
Let $\chi_1(x)$ be a curve with filament function $u_1(x)$ such that $e^{i\xi^2}\hat u_1(\xi)\in H^{\frac 32^+}(0,2\pi)$. If its Frenet frame is well-defined, this condition means the curvature and torsion satisfy:
$$c_1(x)=g(x),\quad \tau_1(x)=\frac x{2}+h(x),\quad g,h\in H^{\frac32^+}_{per}(0,2\pi).$$
Then, there exists $\chi(t)$, with $\chi(1)=\chi_1$, a strong binormal flow solution on $\mathbb R^*$ and weak solution on $\mathbb R$ which generates several corner-singularities at $t=0$. Uniqueness holds in the class of curves with filament functions in the functional frame of Theorem \ref{thblup}.
\end{theorem}
We note that the self-similar solutions $\chi_a$ of the binormal flow  discussed in \S \ref{section-ssbf} enter the framework of this theorem with $g=a$ and $h=0$.

The proof of Theorem \ref{thcritbf} relies in considering first the solutions in Theorem \ref{thblup} with initial data $u_1$ at $t=1$. We obtain a control in time of the $l^{2,\frac 32^+}$-norms of the sequences $\{A_k(t)\}_{k\in\mathbb Z}$ since $e^{i\xi^2}\widehat{u_1}(\xi)\in H^{\frac 32^+}$, and also persistance of this regularity in the limit sequence $\{\alpha_k\}_{k\in\mathbb Z}$. Then we use Theorem \ref{thevpoly} for this sequence $\{\alpha_k\}_{k\in\mathbb Z}\in l^{2,\frac 32^+}$ to obtain the solutions of the binormal flow $\chi(t)$ in Theorem \ref{thcritbf}.

\subsubsection{Sketch of the proof of Theorem \ref{thblup}}
Applying the pseudo-conformal transformation to \eqref{CubicNLS} and to the ansatz \eqref{superp} we look for periodic solutions on $[1,\infty)$ of \eqref{NLSt} with data in $H^r(0,2\pi)$ at $t=1$. 
Thus the modulated Fourier coefficients:
$$B_k(t):=\mathcal F(e^{-it\Delta}v(t))(k),$$
which identifies to $\overline{A_k}(\frac 1t)$, must solve:
$$i\partial_t B_k(t)=\frac{1}{ t}\sum_{k-j_1+j_2-j_3=0}e^{-it(k^2-j_1^2+j_2^2-j_3^2)}B_{j_1}(t)\overline{B_{j_2}(t)}B_{j_3}(t).$$

Let $b>\frac 12$. Following Bourgain's classical approach for 
$$e^{-it\Delta}v(t)(x)=\sum_k B_k(t)e^{ikx},$$
we obtain $H^r$ solutions $v(t)$ of \eqref{NLSt} on $[1,\infty)$ satisfying for any $\nu\in\mathbb N^*, t\in(\nu,\nu+1)$:
$$\|\{B_k(t)\}\|_{l^{2,r}}=\|v(t)\|_{H^r}\leq C\|v\|_{X^{r,b}_\nu}\leq C \|v(\nu)\|_{H^r}=C\|\{B_k(\nu)\}\|_{l^{2,r}},$$
where the $X^{r,b}_\nu$-norm is defined via a smooth localization at $\nu$ of $B$:
$$\|v\|_{X_\nu^{r,b}}:=\left(\int \sum_k \langle k\rangle^{2r} \langle \lambda\rangle^{2b} |\widehat{B}_{k,\nu}(\lambda)|^2d\lambda\right)^\frac 12.$$
Also, the mass is preserved: 
$$M:=\|\{B_k(t)\}\|_{l^{2}}^2=\|v(t)\|_{L^2(0,2\pi)}^2=\|v(1)\|_{L^2(0,2\pi)}^2=\|e^{i\cdot^2}\hat u_1(\cdot)\|_{L^2(0,2\pi)}^2.$$

Now that the solution $v$ is constructed on $[1,\infty)$ we look for further large time properties of $\{B_k(t)\}$. As a first result we obtain a uniform in $k$ pointwise decay for the Fourier in time transform of $\partial_t B_k(t)$. 
\begin{lemma}(Control in time of $\partial_t B_k(t)$)\label{bounddert}
\, For any $\nu\in\mathbb N^*$, and $\eta_\nu$ a smooth cutoff supported in $[\nu,\nu+1]$, we have for any  $\epsilon>0$ the following estimate on the Fourier transform in time:
$$\sup_{k\in \mathbb Z}\sup_{\lambda\in\mathbb R} |\mathcal F(\eta_\nu(\cdot)\partial_t B_k(\cdot))(\lambda)|\leq \frac {C(\epsilon,\|\{B_j(1)\}\|_{ l^{2}}\|\xi^{1^+}\hat\eta_\nu(\xi)\|_{L^\infty})}\nu   \|\{B_j(\nu)\}\|_{l^{2,\epsilon}}.$$
\end{lemma}
\noindent
\begin{proof} We have
$$\mathcal F(\eta_\nu(\cdot)\partial_t B_k(\cdot))(\lambda)=\int e^{it\lambda}\eta_\nu(t)\frac 1t\sum_{NR_k}e^{it(k^2-j_1^2+j_2^2-j_3^2)}(B_{j_1}\overline{B}_{j_2}B_{j_3})(t)dt$$
$$+\int e^{it\lambda}\eta_\nu(t)\frac 1t\left(2M-|B_k(t)|^2\right)B_k(t)dt.$$
Recall that $NR_k$ denotes the noresonant set:
$$NR_k:=\{(j_1,j_2,j_3),\, k-j_1+j_2-j_3=0,\, k^2-j_1^2+j_2^2-j_3^2\neq 0\}.$$
The second term can be estimated straightforwardly to get the $\frac 1\nu$ decay. In the first term we split the summation between the sets:
$$\Lambda_{k,m}=\{(j_1,j_2)\in\mathbb Z^2, (k-j_1)(j_1-j_2)=m\},\quad m\in\mathbb Z^*,$$
and we pass in Fourier in time to get:
$$\int e^{it\lambda}\eta_\nu(t)\frac 1t\sum_{NR_k}e^{it(k^2-j_1^2+j_2^2-j_3^2)}(B_{j_1}\overline{B}_{j_2}B_{j_3})(t)dt$$
$$=\int e^{it\lambda}\eta_\nu(t)\frac 1t\sum_m\sum_{j_1,j_2\in\Lambda_{k,m}}e^{itm}(B_{j_1}\overline{B}_{j_2}B_{k-j_1+j_2})(t)dt$$
$$=\frac 1\nu \int \sum_m\widehat{\tilde \eta}_\nu(\lambda-\lambda_1+\lambda_2-\lambda_3+m) \sum_{j_1,j_2\in\Lambda_{k,m}}\hat B_{j_1,\nu}(\lambda_1)\widehat{\overline{ B}}_{j_2,\nu}(\lambda_2)\hat B_{k-j_1+j_2,\nu}(\lambda_3)d\lambda_1d\lambda_2d\lambda_3.$$
To end the proof of the lemma via the $X^{s,b}_\nu$-norms we perform Cauchy-Schwarz successively:
\begin{itemize}
\item in $j_1,j_2$ using that for all $j_1,j_2\in\Lambda_{k,m}$:
$$\# \Lambda_{k,m}\leq C(\epsilon)m^\epsilon=C(\epsilon)(k-j_1)^{\epsilon}(j_1-j_2)^\epsilon\leq C(\epsilon)\max \{\langle j_1\rangle,\langle j_2\rangle,\langle k-j_1+j_2\rangle\}^{2\epsilon},$$
\item  in $m$ using the decay of $\widehat{\tilde \eta}_\nu$,
\item in $\lambda_1,\lambda_2,\lambda_3$ using the integrability of $\frac 1{\lambda^{2b}}$.
\end{itemize}
\end{proof}

In a more complicated way we obtain the control asymptotically in time  of the low regularity $H^r$ norms, as follows. 
\begin{lemma}(Control in time of the weighted-norms)\label{lemmawev}
For any $t\geq 1, r \in(0,\frac 12)$ we have:
$$\|\{B_k(t)\}\|_{l^{2,r}}\leq C( r,\|\{B_k(1)\}\|_{ l^{2,r}}).$$
\end{lemma}
\begin{proof}We use the evolution of $l^{2,r}$-norms dictated by the equation of the $B_k$s:
$$\sum_k \langle k\rangle^{2r}|B_k(t)|^2-\sum_k \langle k\rangle^{2r}|B_k(1)|^2$$
$$=\int_1^t \sum_{k;NR_k}(\langle k\rangle ^{2r}-\langle j_1\rangle ^{2r}+\langle j_2\rangle ^{2r}-\langle j_3\rangle ^{2r})e^{i\tau (k^2-j_1^2+j_2^2-j_3^2)}B_{j_1}\overline{B_{j_2}}B_{j_3}\overline{B_{k}}(\tau)\frac{d\tau}{\tau},$$
by performing integrations by parts:
$$\sum_k \langle k\rangle^{2r}|B_k(t)|^2\leq\sum_k \langle k\rangle^{2r}|B_k(1)|^2+ C\left[\sum_{k;NR_k}|\varphi_{k,j_1,j_2,j_3}B_{j_1}\overline{B_{j_2}}B_{j_3}\overline{B_{k}}(\tau)\frac 1\tau|\right]_1^t$$
$$+C\left|\int_1^t \sum_{k;NR_k}\varphi_{k,j_1,j_2,j_3}e^{i\tau  (k^2-j_1^2+j_2^2-j_3^2)}B_{j_1}\overline{B_{j_2}}B_{j_3}\overline{\partial_{\tau} B_{k}}(\tau)\frac{d\tau}{\tau}\right|,$$
where 
$$\varphi_{k,j_1,j_2,j_3}:=\frac{\langle k\rangle ^{2r}-\langle j_1\rangle ^{2r}+\langle j_2\rangle ^{2r}-\langle j_3\rangle ^{2r}}{k^2-j_1^2+j_2^2-j_3^2}.$$
We prove that 
$$\sum_{k;(j_1,j_2,j_3)\in NR_k} |\varphi_{k,j_1,j_2,j_3}|N_{j_1}N_{j_2}N_{j_3}P_k\leq C(r,\|N\|_{l^2}) \|N\|_{l^{2,r}}^{0^+}\|P\|_{l^2},$$
for $r<\frac 12$ and any two sequences $\{N_j\}$ and $\{P_j\}$ of positive numbers. The estimates in the cases $r\geq \frac 12$ were obtained in Lemma 2.5 in \cite{ForlanoSeong2022}. This settles the boundary term. For the integral term we develop $\partial_\tau B_k$ using its equation, which yields a six-product term, then use a $\nu-$partition of unity in $\tau$, pass to Fourier in $\tau$, and split the $j$'s summation into different regions where we use different arguments to get the control of the Lemma.
\end{proof}
Now we shall identify the asymptotics of $B_k(t)$. We first obtain for all $k\in\mathbb Z$ a first modulated limit:
$$\exists\,\beta_k:=\lim_{t\rightarrow\infty}\tilde B_k(t),\quad |\tilde B_k(t)-\beta_k|\leq  \frac{C(\|\{B_j(1)\}\|_{ l^{2,r}})}{t},$$
where
$$\tilde B_k(t):=e^{i2M\log t-i\int_1^t|B_k(\tau)|^2\frac{d\tau}{\tau} }B_k(t).$$
To do so we integrate in time the evolution law of $\tilde B_k$, we use again integrations by parts, partition of unity in time, Fourier in time, splitting the nonresonant set $NR_k$ into the sets $\Lambda_{k,m}$, and the two previous lemmas. 
Then we obtain
 $$\exists\,\alpha_k:=\lim_{t\rightarrow\infty}e^{i(2M-|\beta_k|^2)\log t }B_k(t),$$
 $$|B_k(t)-\alpha_ke^{i(2M-|\alpha_k|^2)\log t }|\leq  \frac{C(\|\{B_j(1)\}\|_{ l^{2,s}})}{t}.$$


\section{Appendix: Passing from 1D cubic Schr\"odinger equation to binormal flow}\label{section-Hasimoto}
Hasimoto's initial transformation in 1972, recalled in \S 1.2, is based on the Frenet system of a curve:
$$\left(\begin{array}{c}
T\\n\\b
\end{array}\right)_x=
\left(\begin{array}{ccc}
0 & c & 0 \\ -c & 0 & \tau \\ 0 &  -\tau & 0 
\end{array}\right)
\left(\begin{array}{c}
T\\n\\b
\end{array}\right),$$
where $(T,n,b)$ is the orthonormal frame of the tangent, normal and binormal vectors, $c$ represents the curvature and $\tau$ the torsion (see the book of Spivak \cite{Spivak1970}). His construction is valid for curves with non-vanishing curvature. This constraint was removed in 1997 by Koiso in \cite{Koiso1997} by using instead the notion of parallel transport of frames. 
Before explaining in \S 7.0.2 how the construction works we start with the (time-independent) notion of parallel transport of frames.
  
\subsubsection{Parallel transport of frames and the filament functions of a curve.}\label{sectfil}
We call filament function of a curve $\chi$ a function $u_\chi$ obtained by the following parallel transport procedure. Denote the tangent vector $T=\partial_x\chi$. We consider the parallel frames $(T,e_1,e_2)(x)$ obtained by solving the ODEs: 
$$\partial_xe_1(x)=-\langle \partial_xT,e_1\rangle T,\quad  \partial_xe_2(x)=-\langle \partial_xT,e_2\rangle T,$$ 
with data at some $x_0\in\mathbb R$ given by an orthonormal frame of $\mathbb R^3$. This is always possible if the curve $\chi$ is regular enough, for instance if $\partial_x^2\chi\in L^2_{loc}$. Indeed, this gives global existence in $H^1_{loc}$ for the above ODEs. In addition, the orthonormal frame nature of $(T,e_1,e_2)(x)$ is preserved, since the matrix of the system of evolution in space of $(T,e_1,e_2)(x)$ is antisymmetric.  Then we define: 
$$u_\chi=\langle \partial_x T,e_1\rangle +i\langle \partial_x T,e_2\rangle.$$ 
The real and imaginary part of $u_\chi$ are a normal developement of the curve $\chi$ (see \cite{Bishop1975}). We called them also filament function by refeering to the notion introduced by Hasimoto \cite{Hasimoto1972} for curves $\chi$ with non-vanishing  curvature:
$$c(x)e^{i\int_0^x\tau(s)ds}.
$$
For curves $\chi$ with non-vanishing curvature this function coincides with $u_\chi$ and 
\begin{equation}\label{geo_g}
(e_1+ie_2)(x)=(n+ib)(x)e^{i\int_{x_0}^x \tau(s)ds},
\end{equation}
if in the construction of $u_\chi$ the initial orthonormal frame $(T,e_1,e_2)(x_0)$ is chosen to be the Frenet frame $(T,n,b)(x_0)$. 
Observe that even if the curvature vanishes, the expression in the right-hand side of \eqref{geo_g} continue to make sense via the parallel frame construction.

We note that the only degree of freedom in the filament function $u_\chi$ construction is rotating the initial data $(T,e_1,e_2)(x_0)$, i.e. rotating $(e_1,e_2)$ in the plane orthogonal to $T(x_0)$, which yields by this construction another filament function that is of type $u_\chi(x)e^{i\theta}$ (and changing $x_0$ boils down to the same argument).

We also note that $u_\chi$ is constructed exclusively from $T=\partial_x\chi$ so it does not depend on translations in space of $\chi$, i.e. it does not depend on $\chi(x_0)$. Moreover, $u_\chi$ is uniquely  determined by $T(x_0)=\partial_x\chi(x_0)$ modulo multiplication by $e^{i\theta}$. Therefore if two curves $\chi,\tilde\chi$ satisfy $\chi(x_0)=\tilde \chi(x_0), \partial_x\chi(x_0)=\partial_x\chi(x_0)$ and $u_\chi=u_{\tilde\chi}$, then $\chi=\tilde\chi$. 

Eventually the ODE satisfied by the frame is:
\begin{equation}\label{parallel}
\left( 
\begin{array}{l}
T\\e_1\\e_2
\end{array}
\right)_x= \left( 
\begin{array}{lll}
0  & - {\rm Im\,} u_\chi & {\rm Re\,} u_\chi
  \\
   {\rm Im\,}  u_\chi & 0  & 0
  \\
 - {\rm Re\,} u_\chi  &    0  & 0
\end{array}
\right)\left( 
\begin{array}{l}
T\\e_1\\e_2
\end{array}
\right).
\end{equation}

\medskip

\subsubsection{The Hasimoto approach from NLS to the binormal flow}\label{sectHas}
We recall here the classical Hasimoto approach to construct binormal flow solutions by using sufficiently regular solutions of the cubic Schr\"odinger equation \eqref{Hasimoto}: $$
iu_t+u_{xx}+ (|u|^2-f)u=0, 
$$
for $f$ a space-independent function, as for instance the null function. Let $\mathcal B$ be an orthonormal basis of $\mathbb R^3$, $x_0\in\mathbb R$ and $P \in \mathbb R^3$. 
Let us assume that we have a smooth solution $u$ of \eqref{Hasimoto} on an open time interval~$I$, and let $t_0\in I $.  
Starting from $u$, the first step will be to construct orthonormal frames $(T,e_1,e_2)(t,x)$ such that the first vector $T$ is a solution of the Schr\"odinger map.  
In order to construct these frames we solve the ODE in time:
\begin{equation}\label{ODEtrace}
\left( 
\begin{array}{l}
T\\e_1\\e_2
\end{array}
\right)_t(t,x_0)= \left( 
\begin{array}{lll}
0  & - {\rm Im\,} u_x & {\rm Re\,} u_x
  \\
   {\rm Im\,}  u_x  & 0  & - \frac{|u|^2-f}{2} 
  \\
 - {\rm Re\,} u_x  &     \frac{|u|^2-f}{2}  & 0
\end{array}
\right)\left( 
\begin{array}{l}
T\\e_1\\e_2
\end{array}
\right)(t,x_0),\qquad \forall t \in I,
\end{equation}
with initial condition $\mathcal B$ at time $t_0$. The orthonormal frame character is preserved in the evolution due to the fact that the matrix of the system is antisymmetric. Then, for all $t \in I$, we solve the family of ODEs in space:
\begin{equation}\label{ODEx}
\left( 
\begin{array}{l}
T\\e_1\\e_2
\end{array}
\right)_x(t,x)= \left( 
\begin{array}{lll}
0  & - {\rm Im\,} u & {\rm Re\,} u
  \\
   {\rm Im\,}  u & 0  & 0
  \\
 - {\rm Re\,} u  &    0  & 0
\end{array}
\right)\left( 
\begin{array}{l}
T\\e_1\\e_2
\end{array}
\right)(t,x),
 \qquad  \forall x \in \mathbb R,
\end{equation}
with initial condition the frame $(T,e_1,e_2)(t,x_0)$ at the point $x_0$. Using the fact that $u$ is a solution of equation \eqref{Hasimoto} one can prove\footnote{Since $(T,e_1,e_2)(t,s)$ are orthonormal frames we can write for all $t$
$$\left(\begin{array}{c}
T\\e_1\\e_2
\end{array}\right)_t(t,x)=
\left(\begin{array}{ccc}
0 & a & b \\ -a & 0 & c \\ -b &  -c & 0 
\end{array}\right)
\left(\begin{array}{c}
T\\e_1\\e_2
\end{array}\right)(t,x).$$
We also notice that $(a,b,c)(t,x_0)=(-\Im u_x,\Re u_x, -\frac{|u|^2-f}{2})(t,x_0)$. By computing 
and identifying $T_{ts}=T_{st},  e_{1ts}=e_{1st}$ we get a system for $(a,b,c)$ in terms of $u$, that together with the fact that $u$ satisfies \eqref{Hasimoto} allows for identifying $(a,b,c)(t,x)=(-\Im u_x,\Re u_x, -\frac{|u|^2-f}{2})(t,x)$.} that, at least for regular $u$, 
 the ODE \eqref{ODEtrace} is actually valid at any $x \in \mathbb R$:
 \begin{equation}\label{ODEt}
\left( 
\begin{array}{l}
T\\e_1\\e_2
\end{array}
\right)_t(t,x)= \left( 
\begin{array}{lll}
0  & - {\rm Im\,} u_x & {\rm Re\,} u_x
  \\
   {\rm Im\,}  u_x  & 0  & - \frac{|u|^2-f}{2} 
  \\
 - {\rm Re\,} u_x  &     \frac{|u|^2-f}{2}  & 0
\end{array}
\right)\left( 
\begin{array}{l}
T\\e_1\\e_2
\end{array}
\right)(t,x),\qquad  \forall t \in I, x\in\mathbb R.
\end{equation}

It is then easy to see that $T$ solves the Schr\"odinger map equation. Indeed from \eqref{ODEt} and \eqref{ODEx} we obtain:
$$
T \wedge  T_{xx} = T \wedge\left(  {\rm Re\,} u\, e_1 + {\rm Im\,}  u\, e_2 \right) _x= T \wedge  \left(  {\rm Re\,} u_x\, e_1 + {\rm Im\,}  u_x\, e_2 +  {\rm Re\,} u\, e_{1x} + {\rm Im\,}  u\, e_{2x} \right)$$
$$ =     {\rm Re\,} u_x\, e_2 - {\rm Im\,}  u_x\, e_1=T_t.
$$
Moreover, one can easily check
 that $\chi$ defined 
as
\begin{equation}\label{ChiDefinitionVeryFirst}
\chi(t,x): = P + \int_{t_0}^{t} (T \wedge T_{x})(s,x_0) ds + \int_{x_0}^x  T(t, y) dy, 
\end{equation}
is a solution of the binormal flow equation, 
and that $T$ is its tangent vector.  
Indeed, simply take the time derivative of \eqref{ChiDefinitionVeryFirst} and use the divergence form $T_t=(T\wedge T_x)_x$. Finally let us note that $u(t)$ is a filament function of $\chi(t)$.

By using the complex normal vector $N:=e_{1}+ie_{2}$ the equations write in a shorter way:
\begin{equation}\label{ev}
T_x=\Re{\overline{u}}N,\quad N_x=-uT,\quad T_t=\Im\overline{u_x}N,\quad N_t=-iu_xT+i\frac{|u|^2-f}{2}N,\quad \chi_t=\Im \overline{u}N.
\end{equation}
In particular, if $u$ is solution of \eqref{Hasimoto} that yields $(T,N)$ by the above construction, then  $u(t,x)e^{i\Phi(t)}$ is a solution of \eqref{Hasimoto} with $f-\Phi$ instead of $f$ that yields $(T,Ne^{i\Phi(t)})$ by the above construction, thus the same first vector, thus the same binormal flow solution.

To conclude, let us give some examples of known, rigorously or through experiments, vortex filament smooth dynamics that correspond to a binormal flow solution, and let us give also the solutions of 1D cubic Schr\"odinger equation \eqref{CubicNLS} from which these binormal flow solutions are constructed via Hasimoto's approach: \\
$\bullet$ Straight still lines are parametrized by $\chi(t,x)=(0,0,x)$ and correspond to $u(t,x)=0$,\smallskip\\ 
$\bullet$  Circles translating in the binormal direction (smoke rings) $\chi(t,s)=(\cos s,\sin s,t)$ are obtained from $u(t,x)=e^{-it}$,\smallskip\\ 
$\bullet$  Helices with constant torsion translating on themselves are obtained from $u(t,x)=e^{-it-it\alpha^2+i\alpha x},\alpha\in\mathbb R$,\smallskip\\ 
$\bullet$  Travelling waves solutions of the binormal flow were obtained by Hasimoto from $u(t,x)=e^{-it}e^{-it\alpha^2} e^{i\alpha x}\, \frac{1}{2\sqrt{2}\cosh(x-2\alpha t)}, \alpha\in\mathbb R$.

\bibliographystyle{plain}

\bibliography{biblio}

\end{document}